\documentclass{amsart}
\usepackage{amssymb, amsmath,amsthm }
\usepackage{hyperref}
\usepackage{stmaryrd}
\usepackage[all]{xy}
\newcommand{\cM}{\mathcal{M}}

\newcommand{\cO}{\mathcal{O}}

\newtheorem{prop}{Proposition}[section]

\newtheorem{lem}[prop]{Lemma}
\newtheorem{thm}[prop]{Theorem}
\newtheorem{theorem}[prop]{Theorem}
\newtheorem{cor}[prop]{Corollary}

\theoremstyle{definition}
\newtheorem{defi}[prop]{Definition}
\newtheorem{df}[prop]{Definition}
\newtheorem{examp}[prop]{Example}
\newtheorem{remar}[prop]{Remark}
\newtheorem{remark}[prop]{Remark}

\DeclareMathOperator{\That}{\widehat{T}}

\def\reg{\mathrm{reg}}
\def\rig{\mathrm{rig}}
\def\tri{\mathrm{tri}}

\DeclareMathOperator{\Spf}{Spf}

\DeclareMathAlphabet{\mathpzc}{OT1}{pzc}{m}{it}
\DeclareMathOperator{\Aut}{Aut}
\DeclareMathOperator{\End}{End}
\DeclareMathOperator{\Hom}{Hom}

\DeclareMathOperator{\Ind}{Ind}
\DeclareMathOperator{\cInd}{c-Ind}
\DeclareMathOperator{\Res}{Res}

\DeclareMathOperator{\GL}{GL}

\DeclareMathOperator{\Ker}{Ker}
\DeclareMathOperator{\Coker}{Coker}
\DeclareMathOperator{\WD}{WD}
\DeclareMathOperator{\Gal}{Gal}

\def\rank{\mathop{\mathrm{ rank}}\nolimits}
\DeclareMathOperator{\tr}{tr}

\DeclareMathOperator{\gr}{gr}
\DeclareMathOperator{\Spec}{Spec}
\DeclareMathOperator{\mSpec}{m-Spec}

\DeclareMathOperator{\codim}{codim}
\DeclareMathOperator{\supp}{Supp}

\DeclareMathOperator{\Mod}{Mod}
\DeclareMathOperator{\Alg}{Alg}

\DeclareMathOperator{\Ext}{Ext}

\newcommand{\cIndu}[3]{\cInd_{#1}^{#2}{#3}}

\newcommand{\Indu}[3]{\Ind_{#1}^{#2}{#3}}

\newcommand{\pF}{\mathfrak{p}_F}

\newcommand{\Q}{\mathbb{Q}}
\newcommand{\Qp}{\mathbb {Q}_p}
\newcommand{\Zp}{\mathbb{Z}_p}
\newcommand{\Qpbar}{\overline{\mathbb{Q}}_p}

\newcommand{\Qbar}{\overline{\mathbb{Q}}_p}

\newcommand{\ZZ}{\mathbb Z}

\newcommand{\QQ}{\mathbb Q}
\newcommand{\Aa}{\mathfrak A}
\newcommand{\Pp}{\mathfrak P}

\newcommand{\mm}{\mathfrak m}

\newcommand{\OO}{\mathcal O}

\newcommand{\TT}{\mathbb T}

\DeclareMathOperator{\wtimes}{\widehat{\otimes}}

\newcommand{\nn}{\mathfrak n}

\newcommand{\pp}{\mathfrak p}

\newcommand{\br}[1]{\llbracket #1\rrbracket}
\newcommand{\qq}{\mathfrak{q}}

\newcommand{\ana}{\mathrm{an}}
\newcommand{\sm}{\mathrm{sm}}

\newcommand{\pro}{\mathrm{pro}}

\newcommand{\alg}{\mathrm{alg}}
\newcommand{\cont}{\mathrm{cont}}

\newcommand{\UU}{\mathrm U}

\newcommand{\Lbar}{\overline{L}}
\newcommand{\rhobar}{\bar{\rho}}
\newcommand{\wt}{\mathrm{wt}}
\newcommand{\univ}{\mathrm{univ}}
\newcommand{\loc}{\mathrm{loc}}
\newcommand{\pst}{\mathrm{pst}}

\title[On the density of supercuspidal points]{On the density of supercuspidal points of fixed regular weight in  local deformation rings 
and global Hecke algebras}

\author{Matthew Emerton and Vytautas Pa\v{s}k\={u}nas}
\thanks{The first author is supported in part by the
		  NSF grants DMS-1303450 and DMS-1601871. The second author is supported in part by the DFG, SFB/TR45.}
\date{\today.}
\begin{document} 
\maketitle

\begin{abstract} We study the Zariski closure in the deformation space
of a local Galois representation of the closed points corresponding to potentially semi-stable representations 
 with prescribed $p$-adic Hodge-theoretic properties. We show in favourable cases that the closure  contains a union of irreducible components of the deformation space. We also study an analogous question for global Hecke algebras. 
\end{abstract}

\section{Introduction} In this paper we study the Zariski closure in the deformation space
of a local Galois representation of the closed points corresponding to potentially semi-stable representations with prescribed properties. We show in a number of cases that the closure contains  a union of irreducible components of the deformation space. Our results are most novel in the case when the family consists of potentially semi-stable representations with fixed  Hodge--Tate weights and their Weil--Deligne representations correspond to  supercuspidal representations via the classical local Langlands correspondence; we call such points supercuspidal. We also obtain similar results in the setting of global Hecke algebras, and thus,
in certain circumstances, namely those to which the results of \cite{HMS} apply, in global
Galois deformation rings as well.

Let us explain our results in more detail. Let $F$ be a finite extension of $\Qp$ and let $G_F$ be its absolute Galois group. Let $L$ be a 
further finite extension of $\Qp$ with the ring of integers $\OO$, a uniformizer $\varpi$ and residue field $k$. Let $\rhobar: G_F \rightarrow \GL_n(k)$ be a continuous representation
 and let $\rho^\square: G_F\rightarrow \GL_n(R^\square)$ be  the universal framed 
 deformation of $\rhobar$. In particular, $R^\square$ is a complete local noetherian
  $\OO$-algebra with residue field $k$. If $x$ is a maximal ideal of $R^\square [1/p]$ then 
  its residue field $\kappa(x)$ is a finite extension of $L$; we denote its ring of integers by $\OO_{\kappa(x)}$. By specializing $\rho^{\square}$ at 
  $x$ we obtain a representation $\rho_x: G_F\rightarrow \GL_n(\kappa(x))$. Moreover, 
  its image is contained in $\GL_n(\OO_{\kappa(x)})$ and by reducing the matrix entries modulo
  the maximal ideal of $\OO_{\kappa(x)}$ we obtain $\rhobar$. If $\rho_x$ is potentially semi-stable then we can associate to it $p$-adic Hodge theoretic data: a multiset of integers   
  $\wt(\rho_x)$ called the Hodge--Tate weights and Weil--Deligne representation 
  $\WD(\rho_x)$, which is a representation of $W_F$ (the Weil group of $F$)
  together with a nilpotent operator satisfying certain compatibilities. The following theorem is a representative  example of our results. 
  
  \begin{thm}\label{A} Assume that $p$ does not divide $2n$ and $\rhobar$ admits a potentially crystalline lift of regular weight, which is 
  potentially diagonalisable. Let $E$ be an extension of $F$ of degree $n$, which is either unramified or totally ramified,  and fix regular Hodge--Tate weights $\underline{k}$.
  Let $\Sigma$ be the subset of maximal ideals of $R^\square [1/p]$ consisting of $x\in \mSpec R^\square [1/p]$, such that $\rho_x$ is potentially semi-stable, $\wt(\rho_x)= \underline{k}$, 
  $\WD(\rho_x)$ is irreducible as a representation of $W_F$ and is isomorphic to an induction from $W_E$ of a $1$-dimensional representation. Then the closure of $\Sigma$ in $\Spec R^{\square}$  contains a non-empty 
 union of irreducible components of $\Spec R^{\square}$. 
  \end{thm} 
  
  \begin{remar} If $\Hom_{G_F}(\rhobar, \rhobar(1))=0$ then the deformation problem of $\rhobar$ is unobstructed and thus $R^{\square}$ is formally smooth over $\OO$, so the theorem
  implies the density of $\Sigma$ in the whole of $\Spec R^{\square}$. 
  \end{remar}
  
  \begin{remar} The proof of the theorem uses a global input: the patched module $M_\infty$ constructed in 
  \cite{six-authors}. In the process we identify $G_F$ with a decomposition group at a prime $\pp$ of an absolute Galois group of a totally real field $K$, and $\rhobar$ as a restriction to $G_F$ of a global automorphic representation of $G_K$. In the general case, we obtain precisely those components of $\Spec R^\square$ which contain a closed point corresponding to the restriction to $G_F$ 
  of a global automorphic representation of the kind considered in \cite[\S 2.4]{six-authors} which is crystalline at $\pp$.
  \end{remar} 

\begin{remar} The assumption $p  \nmid 2n$ comes from our use of the patched 
module $M_{\infty}$ constructed in \cite{six-authors} and the assumptions made there. We also have a version of Theorem \ref{A}, when the field extension $E/F$ is wildly ramified. In that case, the condition on $\WD(\rho_x)$ is more complicated to describe; it can be phrased in terms of Bushnell--Kutzko classification of supercuspidal representations of 
$\GL_n(F)$ in terms of types. If we identify $\GL_n(F)=\Aut_F(E)$, then $\WD(\rho_x)$ is required to correspond to a supercuspidal representation $\pi_x$ under the classical local 
Langlands correspondence, such that $\pi_x$ contains a simple stratum $[\mathfrak A, n', 0, \alpha]$ with $E=F[\alpha]$. Moreover, it is enough to consider  those supercuspidals where  $\alpha$ is minimal over $F$. 
\end{remar} 

\begin{remar}\label{vac} The assumption that $\rhobar$ admits a potentially crystalline lift of regular weight, which is 
  potentially diagonalisable, also comes from our use of the patched module $M_{\infty}$. It has been conjectured in \cite{BMcycles}  and it will be shown in the forthcoming joint work of Toby Gee and ME that such lift always exists. We refer the reader to \cite[\S 1.4]{BLGGT} for the notion of potential diagonalisability. 
  \end{remar} 

\begin{remar}\label{rem_HS} The first result of this flavour for arbitrary $F$ and $n$ was proved by Hellmann
 and Schraen in \cite{HS}, for the family of crystabelline representations with fixed Hodge--Tate 
 weights (i.e. those representations, which become crystalline after a restriction to the Galois group of a finite abelian extension of $F$.) We also handle this case; in fact our arguments 
 allow us to mix both cases.
\end{remar}

\subsection{An idea of the proof}  The proof of Theorem \ref{A} uses in a crucial way the 
patched module $M_{\infty}$ constructed in \cite{six-authors}. We will recall some properties of $M_{\infty}$:  it is an $R_{\infty}[G]$-module, 
where $R_{\infty}$ is a complete local noetherian $R^{\square}$-algebra with residue field $k$, obtained 
as part of the patching construction, and $G:=\GL_n(F)$. We let $K=\GL_n(\OO_F)$. As part of the construction of $M_\infty$ we know that the action of the group ring $R_{\infty}[K]\subset
R_{\infty}[G]$ on $M_{\infty}$ extends to the action of the completed group ring $R_{\infty}\br{K}$, 
which makes $M_{\infty}$ into a finitely generated $R_{\infty}\br{K}$-module. 
This allows us to equip $M_{\infty}$ with a canonical topology with respect to which the action 
of $R_{\infty}\br{K}$ is continuous. Moreover, $M_{\infty}$ is 
a projective $\OO\br{K}$-module in the category of linearly compact $\OO\br{K}$-modules. 

The proof then breaks up into two parts. In the first part we work in an abstract setting, 
by axiomatising the properties of $M_\infty$ mentioned above, and develop a theory which we call \textit{capture}. This abstract setting applies both to $M_{\infty}$ and to the zero-th completed homology of a symmetric space, when the group is compact at $\infty$, see section \ref{global}. In fact the setting  applies to the completed homology 
of more general symmetric spaces associated to a reductive group $\mathbb G$ defined over~$\QQ$, whenever $\rhobar$ occurs only in one homological degree. Conjecturally this holds whenever 
$\rhobar$ is irreducible and $l_0=0$, see \cite[Conj.\,3.3]{matt_icm}; the invariant $l_0$ vanishes  if for example $\mathbb G$ is semi-simple and has discrete series.

The output of the theory says that for certain well chosen families $\{ V_i\}_{i\in I}$ of finite dimensional representations of $K$ on $L$-vector spaces the $R_{\infty}$-annihilator
of the family of $R_{\infty}$-modules $M_{\infty}\otimes_{\OO\br{K}} V_i$, for $i\in I$, 
annihilates $M_{\infty}$. This abstract theory is then applied with the family $\{V_i\}_{i\in I}$ consisting of 
Bushnell--Kutzko types for supercuspidal representations, tensored with a fixed irreducible 
algebraic representation $W$ of $\Res^F_{\Qp} \GL_n/F$ evaluated at $L$ and restricted to $K$. Properties of $M_{\infty}$ proved in \cite{six-authors} allow us to convert the automorphic 
information given by the Bushnell--Kutzko types and the algebraic representation $W$ into the Galois 
theoretic information defining the set $\Sigma$. 

 If $f, g:\mathbb N\rightarrow \mathbb N$ are functions, which tend to infinity as $n$ goes to infinity, then we write $f\sim g$, if $\log f(n)=\log g(n) + O(1)$ as $n$ tends to infinity. 
The following proposition is a key tool in the abstract setting considered above; see section \ref{sec_growth} for the definition of the $\delta$-dimension.

\begin{prop}\label{growth_intro} Let $K$ be a uniformly powerful pro-$p$ group. Let $M$ be a 
finitely generated $k\br{K}$-module,  let $\dim M$ be the $\delta$-dimension of $M$ as $k\br{K}$-module, and let $\dim K$ be the dimension of $K$ as a $p$-adic manifold. For each $n\ge1$ let $K_n$ be the closed subgroup of $K$ generated by the $p^n$-th powers of the elements of $K$. Then 
$$\dim_k M_{K_n}\sim (K:K_n)^{\frac{\dim M}{\dim K}},$$
 where $M_{K_n}$ denotes the space of $K_n$-coinvariants. 
 \end{prop}

A  weaker version of this Proposition has been proved by Harris in \cite{harris}, see also \cite{harris_err}.

In the second part of the proof carried out in section 
\ref{support} we show:

\begin{theorem}\label{B} The support of $M_{\infty}$ in $\Spec R_{\infty}$
is equal to a union of irreducible components of $\Spec R_{\infty}$.  
\end{theorem}

We follow the approach of Chenevier \cite{Chenevier}
and Nakamura \cite{Nakamura}, who proved that
the crystalline points are Zariski dense in the rigid analytic generic
fibre $(\Spf R^{\square})^{\rig}$
of $\Spf R^{\square}$.  Indeed, we prove the following result, 
which implies the preceding theorem.

\begin{theorem}\label{C} The Zariski closure in $(\Spf R_{\infty})^{\rig}$
of the set of crystalline points lying in the support of $M_{\infty}$
is equal to 
a union of irreducible components of $(\Spf R_{\infty})^{\rig}$.  
\end{theorem}

The argument of Chenevier and Nakamura uses ``infinite fern''-style
arguments.  In order to modify these arguments so as to be sensitive
to the support of $M_{\infty}$, we use the locally analytic Jacquet
module of \cite{Jacquet-one}, which provides a representation-theoretic
approach to the study of finite slope families.   Breuil, Hellmann,
and Schraen \cite{BHS} have made a careful study of the theory
of the Jacquet module as it applies to $M_{\infty}$, and we use
their results.

The argument outlined above proves the analog of Theorem \ref{A} with $R^{\square}$ replaced by $R_{\infty}$. The statement for $R^{\square}$ is deduced from this by a commutative algebra argument using that   $R_\infty\cong R^{\square}\wtimes_{\OO} A$, where $A$ is a complete local noetherian $\OO$-algebra, which is $\OO$-flat.

\subsection{Historical remarks} Since we have been talking about the results of this paper 
since the workshop on the $p$-adic Langlands correspondence in the Fields  Institute in 2012, some historical comments seem to be in order. 
At the beginning our focus was to prove such density theorem for the global deformation ring 
of an absolutely irreducible representation $\rhobar: \Gal(\overline{\QQ}/ \QQ)\rightarrow \GL_2(k)$, 
which we assume to be modular. The strategy was to apply representation theoretical techniques in the 
spirit of Proposition \ref{growth_intro} to the completed homology of modular curves localised at the ideal of the Hecke algebra corresponding to $\rhobar$. So the  first part of the proof  
was essentially the same as it is now. For the second part of the proof one needed to show that the completed homology of modular curves localised at the ideal of the Hecke algebra corresponding to $\rhobar$ is supported on a union of irreducible components of the global deformation ring of $\rhobar$, which amounted to producing a lower bound on the dimension of the Hecke algebra, see the discussion of
\cite[\S 3.1.1]{matt_icm}. The bound was obtained by showing that the Krull dimension of the Hecke
algebra is at least the $\delta$-dimension of the completed homology, which is equal to $5$, minus the 
$\delta$-dimension of the fibre, which we knew to have dimension $1$: the fibre is determined  by results 
of \cite{emfm} (see also \cite[Lem.\,5.8]{padicJL} for a different argument based on the results of 
\cite{cmf}), and its  $\delta$-dimension is equal to $1$; our original argument is explained in the proof of
\cite[Lem.\,5.14]{padicJL}. 

Such arguments have since appeared in the paper of Gee--Newton \cite{gee-newton}
in a more general context; in particular they also prove a kind of converse theorem called `miracle flatness', which we did not know at the time, by working systematically with Cohen--Macaulay modules. 

We also proved density results 
for the local deformation  rings of $\rhobar: \Gal(\Qbar/\Qp)\rightarrow \GL_2(k)$ by using results of \cite{cmf} and the $p$-adic Langlands correspondence for $\GL_2(\Qp)$.
 
 Then a paper of Hellmann--Schraen appeared \cite{HS}, proving the density result 
 discussed in Remark \ref{rem_HS}. While studying their paper, we realised that the second part of our proof 
 can be obtained using the `infinite fern' arguments by constructing what is now known as the patched 
 eigenvariety. The construction of this patched eigenvariety appeared in the paper by 
 Breuil--Hellmann--Schraen \cite{BHS}, although both ME and Hellmann explained their  ideas 
 to VP independently  at around the same time. 
 
Finally a paper of Scholze \cite{scholze} has appeared, in which he has a much simpler proof of one of the consequences of our theory, discussed in \cite[3.3.2]{matt_icm}. We adapt and explain his argument in section \ref{sec_scholze}. And although our abstract theory is more general, the two key arithmetically interesting examples are covered by this simpler argument due to Scholze, which avoids counting dimensions, required by our criterion for capture. 

In  recent preprints \cite{shen-ning}, \cite{shen-ning2} Shen-Ning Tung has used our result on the support 
of $M_{\infty}$ to obtain a new proof  of the Fontaine--Mazur conjecture for $\GL_2/\QQ$, 
 obtaining new cases of the conjecture when $p=3$ and $p=2$. This was one of the motivations to finally finish this paper. 

\subsection{The structure of the paper} In section  \ref{generalities} we develop the theory of capture in 
an abstract setting. Some of the results of section \ref{sec_capture} have appeared in  \cite[\S 2.4]{CDP}, but 
we recall the proofs for the convenience of the reader.  The main result is Proposition \ref{criterion}, 
which establishes a numerical criterion for a family $\{V_i\}_{i\in I}$ of smooth irreducible representations on finite 
dimensional $L$-vector spaces of a compact $p$-adic analytic group $K$, which does not contain any torsion, to \textit{capture} every projective linearly compact $\OO\br{K}$-module in the sense of 
Definition \ref{defi_capture}. The key Proposition which establishes the link between the 
growth of coinvariants and dimension is proved in section \ref{sec_growth}.

In section \ref{examples} we exhibit two examples where our general theory applies, by taking $K$ 
to be a sufficiently small compact open subgroup of $\GL_n(F)$ and, in the supercuspidal case in section \ref{sec_super},  the family $\{V_i\}_{i\in I}$ inductions of simple characters in the sense of Bushnell--Kutzko theory. The main 
work is then to verify that our numerical criterion holds, which amounts to counting types and their dimensions. We also handle the principal series case in section \ref{sec_prince}.

In section \ref{sec_scholze} we explain that both of these examples can be dealt with in an easier way 
by adopting an argument of Scholze in \cite{scholze}. This argument still uses Bushnell--Kutzko theory, but avoids the counting arguments. 

In section \ref{sec_appl} we apply the theory developed above to the zero-th completed homology of a symmetric space, when the group is compact at infinity. 

In section \ref{support}  we show that the support of $M_{\infty}$ is equal to a union of irreducible components of~$R_\infty$. 

\subsection{Acknowledgments} This paper has its germ in conversations between the authors during the Workshop on Galois Representations and Automorphic Forms at the IAS in March 2011. 
We would like to thank the organisers for the invitation. VP would like to thank Michael Harris and Peter Schneider for their 
comments after his talk in January 2014, Repr\'esentations des groupes r\'eductifs $p$-adiques et applications, Luminy and Eugen Hellmann for explaining his work during his visit to Essen. The authors thank Toby Gee for his detailed comments on an earlier version of the paper, Ga\"etan Chenevier for answering our questions about his work,
Gabriel Dospinescu for pointing out a simplification in the proof of Proposition~\ref{growth},
Yongquan Hu for pointing out a gap in the proof of the same proposition 
in an earlier draft,
and Shen-Ning Tung for his comments on section 6.

\section{Generalities}\label{generalities}

\subsection{Capture}\label{sec_capture}
Let $K$ be a pro-finite group with an open pro-$p$ subgroup. 
Let $\OO\br{K}$ be the completed group algebra, and 
let $\Mod^{\pro}_K(\OO)$ be the category of compact linear-topological $\OO\br{K}$-modules. Let $\{V_i\}_{i\in I}$ be a family of
continuous representations 
of $K$ on finite dimensional $L$-vector spaces, and let $M\in \Mod^{\mathrm{pro}}_K(\OO)$. We will denote by $V_i^*$ the $L$-linear dual of $V_i$.

\begin{defi}\label{defi_capture}We say 
that $\{V_i\}_{i\in I}$ \textit{captures  $M$} if the smallest quotient
$M\twoheadrightarrow Q$ such that
 $\Hom_{\OO\br{K}}^{\rm cont}(Q, V_i^*)\cong \Hom_{\OO\br{K}}^{\rm cont}(M, V_i^*)$, for all $i\in I$, is equal to  $M$. 
\end{defi}

The following equivalent characterizations of capture can be easily verified, see \cite[\S 2.4]{CDP} for details. 

\begin{lem}\label{intersect} Let $$N=\bigcap_{\phi} \Ker \phi,$$ where the intersection is taken over all 
$\phi\in \Hom_{\OO\br{K}}^{\rm cont}(M, V_i^*)$, for all $i\in I$.
Then $\{V_i\}_{i\in I}$ captures $M$ if and only if $N=0$.
\end{lem}

\begin{lem}\label{capture_banach} Let $M\in \Mod^{\pro}_K(\OO)$ be $\OO$-torsion free, let $\Pi(M):=\Hom_{\OO}^{\rm cont}(M, L)$ be the associated
	$L$-Banach space equipped with its supremum norm. Then 
$\{V_i\}_{i\in I}$ captures $M$ if and only if the image of the evaluation map $$\oplus_i \Hom_K\bigl(V_i, \Pi(M)\bigr)\otimes_L V_i \longrightarrow \Pi(M)$$ is a dense subspace.
\end{lem}

Our next lemmas describe various simple aspects of the notion of capture.

\begin{lem}
\label{torsion escapes}
If $\{V_i\}_{i \in I}$ captures $M$, then $M$ is $\cO$-torsion free.
\end{lem}
\begin{proof}
Any $\cO$-torsion element is in the kernel of any morphism $M\to V_i^*$
for any $i \in I$; the lemma then follows from Lemma~\ref{intersect}.
\end{proof}

\begin{lem}\label{capture_product}
If $\{M_j\}_{j \in J}$ is a collection of objects of
$\Mod^{\pro}_K(\OO)$, 
then the family $\{V_i\}_{i \in I}$ captures the product $\prod_{j \in J}M_j$
if and only if it captures each of the $M_j$.
\end{lem}
\begin{proof}
Noting that $\prod_j M_j$ is $\cO$-torsion free if and only if each
$M_j$ is $\cO$-torsion free, we see from Lemma~\ref{torsion escapes}
that it is no loss of generality, when proving the lemma, to assume
further that each $M_j$ is $\cO$-torsion free.   We note also that
$$\Pi(\prod_j M_j) := \Hom^{\cont}_{\cO}(\prod_j M_j, L) = \widehat{\oplus}_j
\Hom^{\cont}_{\cO}(M_j, L) =: \widehat{\oplus}_j \Pi(M_j),$$
where the completed direct sum is taken with respect to the sup norm on
each summand, and where we are using the notation from the statement
of Lemma~\ref{capture_banach}.
Thus the evaluation map
$\oplus_i \Hom_K\bigl(V_i, \Pi(\prod_j M_j )\bigr) \otimes_L V_i 
\to \Pi(\prod_j M_j)$
may be rewritten as the map
$\oplus_i \widehat{\oplus}_j \Hom_K\bigl(V_i, \Pi(M_j)\bigr) 
\to \widehat{\oplus}_j \Pi(M_j).$
This has dense image if and only if each of the evaluation maps
$\oplus_i \Hom_K\bigl(V_i, \Pi(M_j)\bigr) \to \Pi(M_j)$
has dense image.  The present lemma thus follows from
Lemma~\ref{capture_banach}.
\end{proof}
 
\begin{lem}\label{weight} Let $W$ be a continuous representation of $K$ on a finite dimensional $L$-vector space, and let  $\Theta$ be a $K$-invariant $\OO$-lattice in $W$. 
If $\{V_i\}_{i\in I}$ captures $M\otimes_{\OO} \Theta$ then $\{V_i\otimes_L W\}_{i\in I}$ captures $M$. 
\end{lem}
\begin{proof} For each $i\in I$ we have a natural isomorphism 
\begin{equation}\label{first} 
\Hom_{\OO\br{K}}^{\cont}(M\otimes_{\OO} \Theta, V_i^*)\cong \Hom_{\OO\br{K}}^{\cont}(M, (V_i\otimes_L W)^*).
\end{equation}
Thus if $Q$ is a quotient of $M$, such that $$\Hom_{\OO\br{K}}^{\cont}(M, (V_i\otimes_L W)^*)=\Hom_{\OO\br{K}}^{\cont}(Q, (V_i\otimes_L W)^*),$$ for all $i\in I$, then we deduce from 
\eqref{first} that $$\Hom_{\OO\br{K}}^{\cont}(M\otimes_{\OO} \Theta, V_i^*)=\Hom_{\OO\br{K}}^{\cont}(Q\otimes_{\OO} \Theta, V_i^*),$$ for  all $i\in I$. This implies that
$M\otimes_{\OO} \Theta= Q\otimes_{\OO} \Theta$, as  $\{V_i\}_{i\in I}$ captures $M\otimes_{\OO} \Theta$ by assumption. Since $\Theta$ is a free $\OO$-module we deduce that 
$M=Q$.
\end{proof} 

\begin{lem}\label{faithful} Assume that $\{ V_i\}_{i\in I}$ captures $M$ and  let $\phi\in \End_{\OO\br{K}}^{\rm cont}(M)$. If $\phi$ kills 
$\Hom_{\OO\br{K}}^{\rm cont}(M, V_i^*)$ for all $i\in I$ then $\phi=0$.
\end{lem}
\begin{proof} The assumption on $\phi$ implies that
$$\Hom_{\OO\br{K}}^{\rm cont}(\Coker \phi, V_i^*)\cong \Hom_{\OO\br{K}}^{\rm cont}(M, V_i^*), \quad \forall i\in I.$$ Since $\{V_i\}_{i\in I}$ captures $M$, we 
deduce that $M=\Coker \phi$ and thus $\phi=0$.
\end{proof}

\begin{lem}\label{product} The following assertions are equivalent:
\begin{itemize} 
\item[(i)] $\{V_i\}_{i\in I}$ captures every indecomposable projective in $\Mod^{\pro}_{K}(\OO)$;
\item[(ii)]  $\{V_i\}_{i\in I}$ captures every  projective in $\Mod^{\pro}_{K}(\OO)$;
\item[(iii)] $\{V_i\}_{i\in I}$ captures $\OO\br{K}$.
\end{itemize}
\end{lem}
\begin{proof} This follows from Lemma \ref{capture_product} and the fact that every projective object in $\Mod^{\pro}_K(\OO)$ is isomorphic to a product of indecomposable projective objects, which in turn are precisely the indecomposable 
direct summand of $\OO\br{K}$. A version of this argument appears in \cite[Lem.\,2.11]{CDP}.
\end{proof}

\begin{lem}\label{weight_proj} Let $W$ be a continuous representation of $K$ on a finite dimensional $L$-vector space. If $\{V_i\}_{i\in I}$ captures every  projective in $\Mod^{\pro}_{K}(\OO)$
then the same holds for $\{V_i \otimes_L W\}_{i\in I}$. 
\end{lem}

\begin{proof} Let $\Theta$ be a $K$-invariant $\OO$-lattice in $W$. If $K_n$ is an open normal subgroup of $K$, which acts trivially on $\Theta/\varpi^n$,  then 
the map $g \otimes v \mapsto  g \otimes g^{-1} v$, where $g\in K/K_n$ and $v \in \Theta/\varpi^n$ induces an isomorphism of $K$-representations:
$$ \OO/\varpi^n [ K/K_n] \otimes \Theta/\varpi^n \cong \OO/\varpi^n [ K/K_n] \otimes \Theta/\varpi^n,$$
where $K$ acts diagonally on the left hand side and only on the first factor on the right hand side of the isomorphism. By passing to a limit we deduce that 
 $$\OO\br{K}\otimes_{\OO} \Theta\cong \OO\br{K}^{\oplus d}$$  as $\OO\br{K}$-module, where $d=\dim_L W$. 
In particular,  $\OO\br{K}\otimes_{\OO} \Theta$ is projective in $\Mod^{\pro}_{K}(\OO)$, and hence is captured by $\{V_i\}_{i\in I}$ by assumption. Lemma \ref{weight} implies that 
$\{V_i\otimes_L W\}_{i \in I}$ captures $\OO\br{K}$ and the assertion follows from Lemma \ref{product}.
\end{proof}

\subsection{Density}

Let $K$ and $M$ and  $\{V_i\}_{i\in I}$ be as above, but we additionally assume that $K$ is $p$-adic analytic. In the applications $K$ will be an open subgroup of $\GL_n(F)$, where $F$ is a finite extension of $\Qp$.
 Let $R$ be a complete local noetherian $\OO$-algebra with residue field $k$. We assume that we are given a faithful, continuous
$\OO$-linear action of $R$ on $M$, which commutes with the action of $K$. In other words, we assume that $M$ is a compact $R\br {K}$-module, such that the action of $R$ is faithful.
 For each $i\in I$ let $\mathfrak a_i$ be the $R$-annihilator of $\Hom_{\OO\br{K}}^{\rm cont}(M, V_i^*)$. 
 
 \begin{lem}\label{caught_it} If $\{V_i\}_{i\in I}$ captures $M$ then $\bigcap_{i\in I} \mathfrak a_i= 0$.
 \end{lem} 
 \begin{proof}  This  follows from Lemma \ref{faithful}.
 \end{proof}
 
 Let $\mSpec R[1/p]$ be the set of maximal ideals of $R[1/p]$. If $x\in \mSpec R[1/p]$ then its residue field $\kappa(x)$ is a finite extension of $L$, and we denote by $\OO_{\kappa(x)}$ 
 the ring of integers of $\kappa(x)$. Then $\OO_{\kappa(x)}\otimes_R M$ is a compact $\OO\br{K}$-module and 
we define 
 $$\Pi_x:=\Hom^{\cont}_{\OO}(\OO_{\kappa(x)}\otimes_R M, L).$$
 Then $\Pi_x$ equipped with the supremum norm is a unitary $L$-Banach space representation of $K$. 
 
 \begin{prop}\label{dense_cap} Assume that $M$ is a finitely generated $R\br{K}$-module,
that $\{V_i\}_{i \in I}$ captures $M$,
 and that there exists an integer $e$ such that $(\sqrt{\mathfrak a_i})^e \subset \mathfrak a_i$ for all $i\in I$. Then the set 
 $$\Sigma:=\{ x\in \mSpec R[1/p]: \Hom_K(V_i, \Pi_x)\neq 0 \text{ for some } i\in I\}$$ 
 is dense in $\Spec R$. 
 \end{prop}
 \begin{proof} For each $i\in I$ we choose a $K$-invariant bounded $\OO$-lattice $\Theta_i$ in $V_i$.
  It follows from 
 \cite[Prop.\,2.15]{duke} that $M(\Theta_i):=(\Hom_{\OO\br{K}}^{\cont}(M, \Theta_i^d))^d$, where 
 $(\ast)^d:=\Hom_{\OO}^{\cont}(\ast, \OO)$, is a finitely generated $R$-module. Moreover, it follows from 
 \cite[Eqn.\,(11) and  Rem.\,2.18]{duke} that  for a fixed $i\in I$ the annihilator of $M(\Theta_i)$ in $R$ is equal to $\mathfrak a_i$ and it follows from \cite[Prop.\,2.22]{duke} that the set 
 $$\Sigma_i:=\{ x\in \mSpec R[1/p]: \Hom_K(V_i, \Pi_x)\neq 0\}$$
 is precisely $\mSpec (R/\mathfrak a_i)[1/p]$. Since $M(\Theta_i)$ is $\OO$-torsion free, $R/\mathfrak a_i$ is $\OO$-torsion free, and thus the intersection of all maximal ideals in $\Sigma_i$ with $R$ is equal to $\sqrt{\mathfrak a_i}$. Thus the intersection of all maximal ideals in $\Sigma$ with $R$ is equal to $\cap_{i\in I}  \sqrt{\mathfrak a_i}$. 
 Note that $R$ is $\OO$-torsion free since by assumption it acts faithfully on an $\OO$-torsion free module $M$. 
 Our assumption $(\sqrt{\mathfrak a_i})^e \subset \mathfrak a_i$ for all $i\in I$
 implies that the $e$-th power of this ideal is contained in $\cap_{i\in I}  \mathfrak a_i$, which is zero by Lemma \ref{caught_it}. Thus $\Sigma$ is Zariski dense in $\Spec R$.
 \end{proof}
 
 \begin{remar}\label{rem:semi-simple}
 If $M$ is a finitely generated $\OO\br{K}$-module then $\Hom_{\OO\br{K}}^{\rm cont}(M, V_i^*)$ is a finite dimensional $L$-vector space. 
 If for each $i\in I$ the action of $R[1/p]$ on $\Hom_{\OO\br{K}}^{\rm cont}(M, V_i^*)$ is semi-simple then we can conclude that $R/\mathfrak a_i$ is reduced, and 
 hence $\sqrt{\mathfrak a_i}=\mathfrak a_i$, so that the conditions of Proposition \ref{dense_cap} are satisfied.
 \end{remar}
 
 \begin{remar}\label{nonnoetherian} The set-up of Proposition \ref{dense_cap} can be relaxed to allow non-noe\-the\-rian rings as follows. 
 Let $A$ be a projective limit of local artinian $R$-algebras with residue field $k$ with transition maps homomorphisms of local $R$-algebras. Then $A$ is a local 
 $R$-algebra with residue field $k$, which need not be noetherian. Assume that the action of $R\br{K}$ on 
 $M$ extends to a continuous action of $A\br{K}$ for the projective limit topology on $A$.  For each $i\in I$ let $A_i$ be a quotient of $A$, which acts faithfully on $M(\Theta_i)$. 
 Since $R$ is noetherian and $M(\Theta_i)$ is a finitely generated $R$-module, we deduce that the rings $A_i$ are noetherian for all $i\in I$. 
 Thus the proof of Proposition \ref{dense_cap} shows that if $A_i$ are reduced for all $i\in I$ and $A$ acts faithfully on $M$ then $\bigcup_{i\in I} \mSpec A_i [1/p]$
 is dense in $\Spec A$. 
 \end{remar} 
  
 \subsection{Rationality} From the representation theory point of view it is convenient to work over $\Lbar$, the algebraic closure of $L$. However this causes trouble from  the $p$-adic functional analysis point of view, see for example the proof
 of Lemma \ref{capture_product}. The purpose of this section is to deal with this problem.

 \begin{lem}\label{finite_structure} If $V$ is a finite dimensional $\Lbar$-vector space with a continuous $K$-action then there 
 is a finite extension $L'$ of $L$ contained in $\Lbar$ and a finite dimensional $K$-invariant $L'$-subspace $W$ of $V$,
 such that the inclusion $W\hookrightarrow V$ induces an isomorphism $W\otimes_{L'} \Lbar\overset{\cong}{\longrightarrow} V$. 
 \end{lem}
 \begin{proof} We may rephrase the problem as follows: if $\rho:  K \rightarrow \GL_n(\Lbar)$ is a continuous group 
 homomorphism then its image is contained in $\GL_n(L')$ for some finite extension $L'$ of $L$ contained in $\Lbar$. 
 Such statement is proved in Lemma 2.2.1.1 in \cite{bm1}, with $K$ equal to the absolute Galois group of $\Qp$. However, 
 only compactness of the group is used in the proof, so their argument applies in our setting. 
 \end{proof}

 \begin{lem}\label{homs_the_same} Let $L'$ be a finite extension of $L$ contained in $\Lbar$. Let $W$ be a finite dimensional $L'$-vector 
 space with a continuous $K$-action. Then we have a natural isomorphism
 \begin{equation}\label{rat}
 \Hom_{\OO\br{K}}^{\cont}(M, W)\otimes_{L'} \Lbar \overset{\cong}{\longrightarrow} \Hom_{\OO\br{K}}^{\cont}(M, W\otimes_{L'}\Lbar).
 \end{equation}
 \end{lem}
 \begin{proof} Since  $\Hom_{\OO\br{K}}^{\cont}(M, \ast)$ commutes with finite direct sums, the isomorphism in \eqref{rat} holds if we replace $\Lbar$ by  a finite extension $L''$ of $L'$ contained in $\Lbar$. We thus obtain
$$ \Hom_{\OO\br{K}}^{\cont}(M, W)\otimes_{L'} L'' \overset{\cong}{\rightarrow} \Hom_{\OO\br{K}}^{\cont}(M, W\otimes_{L'} L'')
\hookrightarrow \Hom_{\OO\br{K}}^{\cont}(M, W\otimes_{L'} \Lbar).$$
 By passing to a direct limit we obtain an injection 
 $$ \Hom_{\OO\br{K}}^{\cont}(M, W)\otimes_{L'} \Lbar \hookrightarrow \Hom_{\OO\br{K}}^{\cont}(M, W\otimes_{L'} \Lbar).$$
 In order to prove surjectivity we have to show that if $\varphi: M\rightarrow W\otimes_{L'}\Lbar$ is a continuous map of 
 $\OO$-modules then its image is contained in $W\otimes_{L'}L''$ for some finite extension $L''$ of $L'$. This can be proved in the same way as Lemma 2.2.1.1 in \cite{bm1}
 using the compactness of $M$.
 \end{proof}
 
  Let $\{V_i\}_{i\in I}$ be a family of
continuous representations 
of $K$ on finite dimensional $\Lbar$-vector spaces, and let $M\in \Mod^{\mathrm{pro}}_K(\OO)$ as in the previous section. 
We make the analogous definition to Definition \ref{defi_capture}.
\begin{defi}\label{capture_alg} We say 
that $\{V_i\}_{i\in I}$ \textit{captures  $M$} if the smallest quotient
$M\twoheadrightarrow Q$, such that
 $\Hom_{\OO\br{K}}^{\rm cont}(Q, V_i^*)\cong \Hom_{\OO\br{K}}^{\rm cont}(M, V_i^*)$ for all $i\in I$ is equal to  $M$. 
\end{defi}

\begin{lem}   Let $\{V_i\}_{i\in I}$ be a family of continuous representations of $K$ on finite dimensional $\Lbar$-vector spaces. For each $i\in I$ let $L_i$ be a finite extension of $L$ contained in $\Lbar$ and let $W_i$ be a $K$-invariant 
$L_i$-subspace of $V_i$ given by Lemma   \ref{finite_structure}, so that $W_i\otimes_{L_i} \Lbar \cong V_i$ is an isomorphism of $K$-representations. Then $\{V_i\}_{i\in I}$ captures $M$ in the sense of Definition \ref{capture_alg} if and 
only if $\{W_i\}_{i\in I}$ captures $M$ in the sense of Definition \ref{defi_capture}.
\end{lem}
 \begin{proof} The assertion follows from Lemma \ref{homs_the_same}.
 \end{proof}
 
 \subsection{Growth of coinvariants}\label{sec_growth}
 
 In this subsection we let $K$ be a uniformly powerful pro-$p$ group; such groups are discussed in great detail in \cite[\S 4]{analytic}.  We note that uniformly powerful pro-$p$ groups are $p$-adic analytic and every compact $p$-analytic group contains an open uniformly powerful subgroup, so the assumption is harmless.  We denote by  $\dim K$ the dimension of $K$ as a $p$-adic manifold. Finitely generated modules over $k\br{K}$ have a good dimension theory, 
 generalizing the Krull dimension, see \cite{ardakov_brown}, \cite{venjakob}. The purpose of this section is to relate this dimension to the growth of coinvariants by certain open 
 subgroups of $K$.
 
Let $M$ be a finitely generated $k\br{K}$-module. We let 
 $$\codim_{k\br{K}}(M):= \min\{i: \Ext^i_{k\br{K}}(M, k\br{K})\neq 0\},$$  
 $$\dim_{k\br{K}}(M):= \dim K-\codim_{k\br{K}}(M).$$
 We will refer to $\dim_{k\br{K}}(M)$ as the $\delta$-dimension of $M$. For $n\ge 0$ let $K_n:=K^{p^n}$ be the closed subgroup generated by the $p^n$-th powers of elements of $K$. Since $K$ is uniformly powerful, $K_n$ is open in $K$ and we have 
 \begin{equation}\label{index_pro_p}
 (K:K_n)= p^{n(\dim K)}, \quad  \forall n\ge 0.
 \end{equation}
 
 If $f, g:\mathbb N\rightarrow \mathbb N$ are functions, which tend to infinity as $n$ goes to infinity, then we write $f\sim g$, if $\log f(n)=\log g(n) + O(1)$ as $n$ tends to infinity. 
 
 For example, 
if $F$ is a finite extension of $\Qp$, $\UU_F$ is the group of units of its ring of integers, and $\UU_F^n$ is the $n$-th congruence subgroup of $\UU_F$ then 
$(\UU_F: \UU_F^n)\sim q_F^n$, where $q_F$ is the number of elements in the residue field of $F$. 
  \begin{prop}\label{growth} Let $d=\dim_{k\br{K}} M$. Then there are real numbers $a\ge b\ge \frac{1}{d!}$ such that 
  $$ bp^{dn} + O(p^{(d-1)n})\le  \dim_k M_{K_n}\le ap^{dn} + O(p^{(d-1)n}).$$ 
  In particular, 
  $\dim_k M_{K_n}\sim (K: K_n)^{\frac{\dim(M)} {\dim (K)}}$.
  \end{prop}
 \begin{proof} Let $g_1, \ldots, g_m$ be a minimal set of topological generators of $K$, let $\mm$ be the maximal ideal of $k\br{K}$. Then there is an isomorphism of graded rings
 $$ \gr^{\bullet}_{\mm} (k\br{K})\cong k[x_1, \ldots, x_m],$$
 which sends $g_i-1 +\mm^2$ to $x_i$, see \cite[Thm.\,3.4]{ardakov_brown}. Since $K$ is uniformly powerful $g_1^{p^n}, \ldots, g_m^{p^n}$ is a minimal set of generators of $K^{p^n}$ 
 \cite[Thm. 3.6 (iii)]{analytic}. Thus if we let $I_n$ be the closed two-sided ideal of $k\br{K}$ generated by $g_1^{p^n}-1, \ldots, g_m^{p^n}-1$, then 
 \begin{equation}\label{du}
 M_{K_n}\cong M/I_n M \twoheadrightarrow M/\mm^{p^n} M.
 \end{equation}
 Let $A$ and $\widetilde{M}$ be the localizations of $ \gr^{\bullet}_{\mm} (k\br{K})$ and $\gr^{\bullet}_{\mm}(M)$ 
 at the ideal $\gr^{>0}_{\mm} (k\br{K})$ respectively, and let $\nn$ be the maximal ideal of $A$. Then 
 \begin{equation}\label{trys}
\dim_ k M/\mm^n M=\sum_{i=0}^{n-1} \dim_k (\mm^i M/\mm^{i+1} M)= \ell(\widetilde{M}/\nn^n \widetilde{M}).
 \end{equation}
Proposition 5.4 of \cite{ardakov_brown} implies that the Krull dimension of $\gr^{\bullet}_{\mm}(M)$, and hence of $\widetilde{M}$, is equal to $d$. We know 
that for large $n$, $\ell(\widetilde{M}/\nn^n \widetilde{M})$ is a polynomial in $n$ of degree $d$ and leading coefficient $e/d!$, where $e$ is the Hilbert--Samuel multiplicity 
of $\widetilde{M}$, see \cite[\S14]{matsumura}. In particular, 
\begin{equation}\label{keturi}
\ell(\widetilde{M}/\nn^{p^n} \widetilde{M}) = \frac{e}{d!} p^{dn} + O(p^{(d-1)n}).
\end{equation}
The surjection $M \twoheadrightarrow M_{K_n}$ induces a surjection of graded modules $$\gr^{\bullet}_{\mm}(M)\twoheadrightarrow \gr^{\bullet}_{\mm}(M_{K_n}).$$ Since $g_i^{p^n}$ act trivially on $M_{K_n}$, the 
surjection factors through 
\begin{equation}\label{penki}
\gr^{\bullet}_{\mm}(M)/(x_1^{p^n}, \ldots, x_m^{p^n})\gr^{\bullet}_{\mm}(M) \twoheadrightarrow \gr^{\bullet}_{\mm}(M_{K_n}).
\end{equation}
Let $\nn^{[p^n]}$ be the ideal of $A$ generated by $x^{p^n}$ for all $x\in \nn$. If $a_1, \ldots, a_m$ is a system of parameters of $A$ then $\nn^{[p^n]}=(a_1^{p^n}, \ldots, a_m^{p^n})$ as $A$ is a regular ring. 
Hence, the localization of $\gr^{\bullet}_{\mm}(M)/(x_1^{p^n}, \ldots, x_m^{p^n})\gr^{\bullet}_{\mm}(M)$ at $ \gr^{>0}_{\mm} (k\br{K})$ is equal to $\widetilde{M}/\nn^{[p^n]} \widetilde{M}$ and we obtain
\begin{equation}\label{sesi}
\dim_k (\gr^{\bullet}_{\mm}(M)/(x_1^{p^n}, \ldots, x_m^{p^n})\gr^{\bullet}_{\mm}(M) )= \ell(\widetilde{M}/\nn^{[p^n]} \widetilde{M}).
\end{equation}
The main result of \cite{monsky} says that there is a real number $e'$, called the Hilbert--Kunz multiplicity,  such that 
 \begin{equation}\label{septyni}
 \ell(\widetilde{M}/\nn^{[p^n]} \widetilde{M})=  e' p^{dn} + O(p^{(d-1)n}).
 \end{equation}
Actually, as G.\,Dospinescu pointed out to us, we don't require the full strength 
of~\cite{monsky}.  Indeed, since $\mathfrak m$ is generated by $m$ generators,
we have  $\mathfrak \nn^{m p^n} \subseteq  \nn^{[p^n]},$
and so (just using the theory of the Hilbert--Samuel multiplicity),
we find that
 \begin{equation}
\label{astuoni}
 \ell(\widetilde{M}/\nn^{[p^n]} \widetilde{M})
\leq \ell(\widetilde{M}/\nn^{m p^n} \widetilde{M}) = \dfrac{e}{d!}m^d p^{dn}
   + O(p^{(d-1)n}),
 \end{equation}
which suffices for our purposes.

In any event,
we conclude from \eqref{du}, \eqref{trys}, \eqref{penki}, and \eqref{sesi}
that
$$\ell(\widetilde{M}/\nn^{p^n} \widetilde{M})\le \dim_k M_{K_n} \le
\ell(\widetilde{M}/\nn^{[p^n]} \widetilde{M}).$$
The assertion then follows from
\eqref{keturi} and \eqref{septyni} (or its variant~\eqref{astuoni}).  \end{proof}

 \begin{remar}Proposition \ref{growth} asserts that \begin{equation}
\label{assymp} \dim_k M_{K_n}\sim (K: K_n)^{\frac{\dim(M)} {\dim (K)}}.
\end{equation} However, we caution the reader that the analogous statement
might not hold for an arbitrary filtration of $K$ by open normal subgroups. For
example, let $K=\Zp \oplus \Zp$ and  let $M=k\br{\Zp}$, on which $K$ acts via
the projection onto the second factor, let $K_n= p^n \Zp \oplus p^{2^n} \Zp$.
Then $\dim_k M_{K_n}= p^{2^n}$ and $(K:K_n)= p^{n+2^n}$, $\dim M=1$ and $\dim
K=2$, so that \eqref{assymp} does not hold. One should modify the statement of
\cite[Thm.\,1.18]{calegari-emerton} accordingly. The authors of
\cite{calegari-emerton} claim that a stronger version of Proposition \ref{growth} is proved in
\cite[\S 5]{ardakov_brown}, however we have not been able to verify this.
\end{remar}
 
 \begin{remar}  A weaker version of this statement follows from Theorem 1.10 in
\cite{harris}, see also \cite{harris_err}.
\end{remar}

 \subsection{Numerical criterion for capture} From now on we assume that $K$ is
a $p$-adic analytic group, which doesn't have any torsion.   This allows for
more general groups $K$ than considered in the previous section. As already
noted every such $K$ contains an open subgroup $K'$, which is a uniformly
powerful pro-$p$ group, see \cite[Thm.\,8.23]{analytic}. For a  finitely
generated $\OO\br{K}$-module $M$ we let $\dim_{\OO\br{K}}(M):=
\dim_{\OO\br{K'}}(M)$. This definition does not depend on the choice of $K'$,
see \cite[\S 5.4]{ardakov_brown}. Since the identity in $K$ has a basis of open
neighbourhoods consisting of open uniformly powerful pro-$p$ subgroups, this
implies that $\dim_{\OO\br{K}}(M)= \dim_{\OO\br{K''}}(M)$ for any open subgroup
$K''$ of $K$. 
 
 Our assumption on $K$ implies that $K$ is pro-$p$ and the completed group
algebras $k\br{K}$ and  $\OO\br{K}$ are integral domains, see
\cite[Thm.\,4.3]{ardakov_brown}.

\begin{lem}\label{dim_cyclic} Let  $M$ be a cyclic $k\br{K}$-module. If
$\dim_{k\br{K}}(M)= \dim K$ then the surjection $k\br{K}\twoheadrightarrow M$
is an isomorphism.  \end{lem} \begin{proof} Let $\Lambda=k \br{K}$. Then
$\dim_{\Lambda} M = \dim K$ if and only if  $\Hom_{\Lambda}(M, \Lambda)\neq 0$.
Since $M$ is cyclic, it is enough to prove that $\Hom_{\Lambda}(\Lambda/
\Lambda y, \Lambda)=0$ for all non-zero $y$ in the maximal ideal of $\Lambda$.
This follows by applying the functor $\Hom_{\Lambda}(\ast, \Lambda)$ to the
exact sequence $\Lambda \overset{\cdot y}{\rightarrow} \Lambda\rightarrow
\Lambda/\Lambda y\rightarrow 0$, and using the fact that our assumptions on $K$
imply that $\Lambda$ is an integral domain, so that multiplication with $y$
induces an injective map from $\Hom_{\Lambda}(\Lambda, \Lambda)$ to itself.
\end{proof}

Let $\{K_n\}_{n\ge 1}$ be a basis of open neighbourhoods of $1$ in $K$
consisting of normal subgroups. We assume that there is $n_0\ge 0$, such that
$K_{n_0}$ is a uniform pro-$p$ group, and for each $m\ge 1$ there is an $n(m)$,
such that $K_{n_0}^{p^m}= K_{n(m)} $.  Let $\mathcal C(K, L)$ be the space of
continuous $L$-valued functions on $K$. Since $K$ is compact the supremum norm
defines a Banach space structure on $\mathcal C(K, L)$; it is an admissible
unitary $L$-Banach space representation of $K$ for the $K$-action given by
right translations.

\begin{lem}\label{cap_ban} Let $\Pi$ be a closed $K$-invariant subspace of
$\mathcal C(K, L)$. If $\dim_L \Pi^{K_n} \sim (K:K_n)$ then $\Pi= \mathcal C(K,
L)$.
\end{lem} 
\begin{proof}  Let $\Theta$  be the unit ball in $\Pi$ with respect to the supremum norm on $\mathcal C(K, L)$. The map 
$K\rightarrow \OO\br{K}$, $g \mapsto g$ induces an isomorphism $\mathcal C(K, L)\cong \Hom^{\cont}_{\OO}(\OO\br{K}, L)$ \cite[Cor.2.2]{iw}.
Thus, if we let $\Theta^d:=\Hom_{\OO}(\Theta, \OO)$ then Schikhof's duality induces a surjection of $\OO\br{K}$-modules, $\OO\br{K}\twoheadrightarrow \Theta^d$ \cite[Thm.\,3.5]{iw}. We claim that the surjection is 
an isomorphism. By dualizing back, the claim implies the Lemma. Since $\Theta^d$ is $\OO$-torsion free, it is enough to show that $k\br{K}\twoheadrightarrow \Theta^d/\varpi$ is an isomorphism. 
Lemma \ref{dim_cyclic} implies that it is enough to show that $\dim_{k\br{K}} ( \Theta^d/\varpi)= \dim K$. Since the dimension does not change if we pass to an open subgroup, it is enough 
to show that $\dim_{k\br{K_{n_0}}} ( \Theta^d/\varpi)= \dim K_{n_0}$. Proposition \ref{growth} implies that it is enough to show that $\dim_k (( \Theta^d/\varpi)_{K_n})\sim (K_{n_0} : K_n)$. Since 
$\Theta^d/\varpi \cong (\Theta/\varpi)^{\vee}$, we get that $( \Theta^d/\varpi)_{K_n}\cong (( \Theta/\varpi)^{K_n})^{\vee}$.  Thus 
$$\dim_k (( \Theta^d/\varpi)_{K_n})= \dim_k (( \Theta/\varpi)^{K_n}) \ge \rank_{\OO} \Theta^{K_n}= \dim_L \Pi^{K_n}\sim (K: K_n).$$
On the other hand $$\dim_k (( \Theta^d/\varpi)_{K_n})\le \dim_k k\br{K}_{K_n}=(K:K_n).$$ 
Since $(K: K_n)\sim (K_{n_0}: K_n)$ we obtain the claim. 
\end{proof}

Let $\{V_i\}_{i\in I}$ be a family of pairwise non-isomorphic,  smooth irreducible 
$\overline{L}$-rep\-re\-sen\-ta\-tions of 
$K$.  Let 
$$d(n):= \sum_{i\in I} (\dim_{\overline{L}} V_i^{K_n})^2.$$
We will show in the course of the proof of the next proposition that only finitely many terms in the sum are non-zero, so that $d(n)$ is finite. 

\begin{prop}\label{criterion} If $d(n)\sim (K:K_n)$ then $\{V_i\}_{i\in I}$ captures $\OO\br{K}$ and hence every projective in 
$\Mod^{\pro}_K(\OO)$.
\end{prop}
\begin{proof} Let $\OO\br{K}\twoheadrightarrow M$ be the smallest quotient such that 
\begin{equation}\label{capture_again}
\Hom^{\cont}_{\OO\br{K}}(\OO\br{K}, V_i^*)=\Hom^{\cont}_{\OO\br{K}}(M, V_i^*),\quad  \forall i\ge 0. 
\end{equation}
Then $M$ is $\OO$-torsion free, and $\Pi:=\Hom_{\OO}^{\cont}(M, L)$ is a closed $K$-invariant subspace of $\mathcal C(K, L)$.  Lemma \ref{cap_ban} and Schikhof duality imply that it is enough 
to show that $\dim_L \Pi^{K_n}\sim (K:K_n)$.  Let $I(n):=\{i\in I: V_i^{K_n}\neq 0\}$. 
Since $V_i$ are absolutely irreducible and $K_n$ are normal in $K$, $V_i^{K_n}=V_i$ for all $i\in I(n)$. In particular, such $V_i$ are representations of a finite group $K/K_n$. Thus, $I(n)$ is a finite set and $d(n)$ is finite. 
Thus  for a fixed $n$  there is  a finite extension $L_n$ of $L$, such that all  $V_i$, $i\in I(n)$ are defined over $L_n$. Alternatively, we could use Lemma \ref{finite_structure}.

 Since $(\Pi_{L_n})^{K_n}\cong (\Pi^{K_n})_{L_n}$, 
after extending the scalars to $L_n$ we may assume that $L_n=L$.
Schikhof duality \cite[Thm.3.5]{iw}, together with \eqref{capture_again},
implies that 
\begin{equation}
\Hom_K(V_i, \Pi)\cong \Hom_K(V_i, \mathcal{C}(K, L)) \cong V_i^*, \quad \forall i\in I(n)
\end{equation}
(the second isomorphism holding by Frobenius reciprocity).
This implies that $$\oplus_{i\in I(n)} V_i\otimes V_i^*$$ occurs as a subspace of $\Pi^{K_n}$, and so $\dim_L \Pi^{K_n}\ge d(n)$. Since $\Pi^{K_n}$ is a subspace of 
$\mathcal C(K, L)^{K_n}$, $\dim_L \Pi^{K_n} \le (K:K_n)$.  Since $d(n)\sim (K:K_n)$ by assumption, we deduce that $\dim_L \Pi^{K_n}\sim (K: K_n)$. 
It follows from Lemma \ref{cap_ban}  and Lemma \ref{capture_banach} that $\{V_i\}_{i\in I}$ captures $\OO\br{K}$. The last assertion follows from Lemma \ref{product}.
\end{proof}

\section{Examples of capture}\label{examples}

\subsection{Locally algebraic representations}
Let $G$ be an affine group scheme of finite type over $\Zp$ such that $G_L$ is a split connected reductive group over $L$. Let 
$\Alg(G)$ be the set isomorphism classes of irreducible rational representations of $G_L$, which we view as representations of $G(\Zp)$ via the inclusion $G(\Zp)\subset G(L)$.
\begin{prop}\label{capture_algebraic} $\Alg(G)$  captures every projective object in $\Mod^{\pro}_{G(\Zp)}(\OO)$.
\end{prop}
\begin{proof} \cite[Prop.\,2.12]{CDP}.
\end{proof}

\begin{cor} Let $W$ be any finite dimensional $L$-representation of $G(\Zp)$, such that the action factors through a finite quotient. Then 
$\{V\otimes_L W\}_{V\in \Alg(G)}$ captures every projective object in $\Mod^{\pro}_{G(\Zp)}(\OO)$.
\end{cor} 
\begin{proof} This follows from Proposition \ref{capture_algebraic} and Lemma \ref{weight_proj}.
\end{proof}

\subsection{Subgroups and filtrations} Let $F$ be a finite extension of $\Qp$ with the ring of integers $\OO_F$,  the maximal ideal 
$\pp_F$ and a uniformizer $\varpi_F$.  Let $V$ be an $N$-dimensional $F$-vector space. 
Let $A=\End_F(V)$ and let $G=\Aut_F(V)$.  A lattice chain $\mathcal L$ is a set of $\OO_F$-lattices in $V$, $\{L_i: i\in \ZZ\}$, such that 
$L_{i+1}\subset L_i$ for all $i\in \ZZ$  and there is an integer $e=e(\mathcal L)$, called the period of $\mathcal L$,  such that $ \varpi_F L_i= L_{i+e}$, for  all $i\in \ZZ$. 
To a lattice chain we associate a hereditary order, 
$\mathfrak A:= \{ a\in A: a L_i\subset L_i, \forall i\in \ZZ\}$. If $n\in \ZZ$ then let $\mathfrak P^n:=\{a\in A: a L_{i} \subset L_{i+n}, \forall i\in \ZZ\}$. Then 
$\mathfrak P:=\mathfrak P^1$ is the Jacobson radical of $\mathfrak A$, and $\mathfrak P^m \mathfrak P^n= \mathfrak P^{n+m}$, for all $n, m\in \ZZ$.
Let 
$$ \UU(\mathfrak A):=\UU^0(\mathfrak  A):=\mathfrak A^{\times}, \quad \UU^n(\mathfrak A):= 1+ \mathfrak P^n, \quad \forall n\ge 1.$$
Then $\UU(\mathfrak A)$ is a compact open subgroup of $G$ and $\{\UU^n(\mathfrak A)\}_{n\ge 0}$ is a basis of open neighborhoods of $1$ by normal 
subgroups of $\UU(\mathfrak A)$.
\begin{examp}\label{unramified_order} If $\mathcal L=\{ \varpi_F^i L_0: \forall i\in \ZZ\}$, then $\UU(\mathfrak A)=\Aut_{\OO_F}(L_0)$ and $\UU^n(\mathfrak A)$ is the kernel of the
natural map $\Aut_{\OO_F}(L_0)\twoheadrightarrow \Aut_{\OO_F}(L_0/\varpi^n_F L_0).$
\end{examp} 

\begin{lem}\label{p_power} Let $e(\mathcal L|\Qp):=e(\mathcal L) e(F|\Qp)$ and  if $p>2$ then let $c= e(\mathcal L|\Qp)$, if $p=2$ then let $c=2e(\mathcal L|\Qp)$. 
Then $\UU^n(\mathfrak A)^p= \UU^{n+e(\mathcal L|\Qp)}(\mathfrak A)$ for all $n\ge c$.
\end{lem}
\begin{proof} If $x\in \mathfrak P^n$ then $(1+x)^p\in 1 + p\mathfrak P^n=  \UU^{n+e(\mathcal L|\Qp)}(\mathfrak A)$. This yields one inclusion. For the other 
inclusion, we observe that the series defining $\exp$ converges on $p \mathfrak A=\mathfrak P^c$, if $p>2$, and  on $4 \mathfrak A= \mathfrak P^c$, if $p=2$, to give 
a homeomorhism $\exp: \mathfrak P^c\rightarrow \UU^{c}(\mathfrak A)$, with the inverse map given by the series for $\log$. If $y\in \mathfrak P^{n+e(\mathcal L|\Qp)}$
and $n\ge c$, then $\log(1+y)/p \in \mathfrak P^n$, and hence $\exp(\log(1+y)/p)\in \UU^n(\mathfrak A)$ is a $p$-th root of $1+y$. 
\end{proof}

\begin{prop}\label{uniform} Let $c$ be the constant defined in Lemma \ref{p_power}. If $n\ge c$ then $\UU^n(\mathfrak A)$ is a uniform pro-$p$ group. 
\end{prop}
\begin{proof} If $2n\ge m$ then the map $g\mapsto g-1$ induces an isomorphism of groups
\begin{equation}\label{nice_quotient}
\UU^{n}(\mathfrak A)/\UU^{m}(\mathfrak A)\overset{\cong}{\rightarrow}\mathfrak P^n/\mathfrak P^m,
\end{equation}
where the group operation on the right-hand-side is addition, and so both groups are abelian. Lemma \ref{p_power} implies that the derived subgroup 
$[\UU^n(\mathfrak A), \UU^n(\mathfrak A)]$ is contained in $\UU^n(\mathfrak A)^p$, if $p>2$ and in $\UU^n(\mathfrak A)^4$ if $p=2$. Hence, $\UU^n(\mathfrak A)$ is powerful.  It follows from 
\eqref{nice_quotient} and Lemma \ref{p_power} that $\UU^n(\mathfrak A)/ \UU^n(\mathfrak A)^p$ is finite. Since $\UU^n(\mathfrak A)$ is a pro-$p$
group, this implies that $\UU^n(\mathfrak A)$ is finitely generated. Since $\exp(m a)= \exp(a)^m$, for $m\in \ZZ$,  and $\mathfrak P^n$ is a torsion free $\Zp$-module, the homeomorphism 
$\exp: \mathfrak  P^n \cong \UU^n(\mathfrak A)$ implies that $\UU^n(\mathfrak A)$ is torsion free. Theorem 4.5 of \cite{analytic} implies that $\UU^n(\mathfrak A)$ is uniformly powerful. 
\end{proof}

\subsection{Principal series types for $\GL(N, F)$}\label{sec_prince} We keep the notation of the previous subsection, but fix a basis of $V$ and identify $G=\GL(N, F)$.
Let $\UU_F:=\UU^0_F:=\OO_F^{\times}$ and $\UU^i_F=1+\pp_F^i$ for $i\ge 1$. Let $T$ the subgroup of diagonal matrices in $G$ and 
let $T_i$ be the subgroup of $T$ with all diagonal entries in $\UU^i_F$. 

For $n\ge 1$ let $I^n$ be the inverse image in $\GL(N, \OO_F)$ of unipotent upper-triangular matrices 
in $\GL(N, \OO_{F}/\pp_F^n)$ under the reduction map. Then $I^n$ is a compact $p$-adic analytic group.

\begin{lem} There is $n_0$ such that $I^n$ is torsion free for all $n\ge n_0$.
\end{lem}
\begin{proof} Let $t$ be a diagonal matrix with diagonal entries $t_{ii}= \varpi_F^{(N-i) m}$, $1\le i\le N$,  with $n\ge Nm $. 
A matrix in $g\in t I^n t^{-1}$ is  congruent to $1$ modulo $\varpi_F^m$. It follows from the last part 
of the proof of Proposition \ref{uniform} that if $m$ is large enough then 
$g$ cannot be a torsion element unless it is equal to the identity. 
\end{proof} 
  
We choose $n_0$, such that $I^{n_0}$ is torsion free and $n_0> N$.  We say that a smooth character $\chi: \UU^{m}_F\rightarrow \Lbar^{\times}$ has conductor $i$, if $\chi$ is trivial on $\UU^{i+1}_F$ and nontrivial on $\UU^i_F$. 
For $n\ge n_0$ let $X(n)$ be the set of ordered $N$-tuples $(\chi_1, \ldots, \chi_N)$ of  smooth characters of $\UU^{n_0}_F$, such that $\chi_i$ has conductor $n-i+1$. Then for $n\ge n_0 +N$ we have
\begin{equation} \label{size_X_n}
|X(n)|= \prod_{i=1}^N (|\UU^{n_0}_F/\UU^{n-i+1}_F| -|\UU^{n_0}_F/\UU^{n-i}_F|)\sim q_F^{nN}.
\end{equation}
Let 
$$\mathfrak J^n:=\begin{pmatrix}\pF^n & \OO_F & \ldots &&  \OO_F\\
                                                          \pF^n  & \pF^{n-1}&\OO_F& \ldots & \OO_F\\ 
                                                          \vdots &   \vdots &    & \ddots         & \vdots\\
                                                          
                                                          \pF^n & \pF^{n-1} & \ldots & &  \OO_F  \\
                                                          \pF^n & \pF^{n-1} & \ldots & & \pF^{n-N+1}\end{pmatrix}.$$

If $x, y\in \mathfrak J^n$ then $x+y$ and  $xy$ lie in $\mathfrak J^n$. Hence, 
$$ J^n:= 1 +\mathfrak J^n$$
is an open subgroup of $G$ normalized by $T^0$.  We may interpret an element of $X(n)$ as a character of $T^{n_0}$, which maps a 
diagonal matrix $(d_1, \ldots, d_N)$ to $\chi_1(d_1)\ldots \chi_N(d_N)$. Given such $\chi\in X(n)$ we may extend it to a character of $T^{n_0} J^{n+1}$ by letting $\chi(tk)=\chi(t)$ for all $t\in T^{n_0}$ 
and $k\in J^{n+1}$. Let 
$$ V_{\chi}:= \Indu{T^{n_0} J^{n+1}}{I^{n_0}}{\chi}.$$ 
 
 \begin{prop}\label{types_prince}  If $\pi$ is a smooth irreducible $\Lbar$-representation of $G$ then 
$$\Hom_{I^{n_0}}( V_{\chi}, \pi)\neq 0 \iff\pi\cong \Indu{B}{G}{\psi}$$ 
with $\psi|_{T^{n_0}}= \chi_1\otimes\ldots\otimes \chi_N$.
 \end{prop} 
 \begin{proof} Let $\tilde{\chi}: T^0\rightarrow \Lbar^{\times}$ be any character of $T^0$, such that $\tilde{\chi}|_{T^{n_0}}=\chi$ and let $W_{\tilde{\chi}}:=\Indu{T^0 J^{n+1}}{K}{\tilde{\chi}}$, 
 where $K=\GL_N(\OO_F)$. It follows from \cite{BK_cover}, see \cite[Prop.\,3.10]{HS}, that $W_{\tilde{\chi}}$ is an irreducible representation of $K$ and is a type for the inertial Bernstein 
 component $[T^0, \tilde{\chi}]$. This implies that  $\Hom_{K}(W_{\tilde{\chi}}, \pi)\neq 0$ if and only if $\pi\cong \Indu{B}{G}{\psi}$ with $\psi|_{T^{0}}= \tilde{\chi}$. Since 
 \begin{equation}\label{induce_up}
 \Indu{I^{n_0}}{K}{V_\chi}\cong \Indu{T^{n_0} J^{n+1}}{K}{\chi}\cong \Indu{T^0J^{n+1}}{K}{(\Indu{T^{n_0} J^{n+1}}{T^0 J^{n+1}}{\chi})}\cong \bigoplus_{\tilde{\chi}} W_{\tilde{\chi}}, 
 \end{equation}
 where the sum is taken over all $\tilde{\chi}: T^0\rightarrow \Lbar^{\times}$, such that $\tilde{\chi}|_{T^{n_0}}= \chi$, the assertion follows from Frobenius reciprocity.
 \end{proof}
 
 \begin{defi} $X:= \cup_{n\ge 1} X(n)$.
 \end{defi}
 
 \begin{lem}\label{irreducible_distinct} The representation $V_{\chi}$ is irreducible. If $\chi_1, \chi_2\in X$ are distinct then $V_{\chi_1}\not \cong V_{\chi_2}$. 
 \end{lem}
 \begin{proof} Let $\chi_1\in X(n_1)$ and $\chi_2\in X(n_2)$ and let $\tilde{\chi}_1$ and $\tilde{\chi}_2$ be characters of $T^0$ extending $\chi_1$, and $\chi_2$, respectively. 
 We use the notation introduced in the proof of Proposition \ref{types_prince}. If $\Hom_K(W_{\tilde{\chi}_1}, W_{\tilde{\chi}_2})\neq 0$ then $W_{\tilde{\chi}_1}\cong W_{\tilde{\chi}_2}$. 
Since these representations are types for the components $[T^0, \tilde{\chi}_1]$ and 
$[T^0, \tilde{\chi}_2]$ respectively, we conclude that $\tilde{\chi}_1=\tilde{\chi}_2$. It follows from
\eqref{induce_up} and Frobenius reciprocity, that if $\chi_1\neq \chi_2$ then $\Hom_{I^{n_0}}(V_{\chi_1}, V_{\chi_2})=0$, and if $\chi_1=\chi_2=\chi$ then 
$$\dim_{\Lbar}\Hom_K( \Indu{T^{n_0}J^{n+1}}{K}{\chi},  \Indu{T^{n_0}J^{n+1}}{K}{\chi})= \dim_{\Lbar} \Hom_K(  \Indu{I^{n_0}}{K}{V_\chi}, \Indu{I^{n_0}}{K}{V_\chi})$$
is equal to the number of characters $\tilde{\chi}: T^0\rightarrow \Lbar^{\times}$, such that $\tilde{\chi}|_{T^{n_0}}= \chi$, which is equal to 
$$ | T^0/ T^{n_0}|=|T^{n_0} J^{n+1}\backslash T^{0} J^{n+1} / T^{n_0} J^{n+1}|.$$
Since $T^0 J^{n+1}$ is contained in the $K$-intertwining of $\chi$, we deduce that the $K$-intertwining of $\chi$ is equal to $T^0 J^{n+1}$. Since $I^{n_0} \cap T^0 J^{n+1}= T^{n_0} J^{n+1}$, 
we deduce that the $I^{n_0}$-intertwining of $\chi$ is equal to $T^{n_0} J^{n+1}$ and so $V_{\chi}$ is irreducible.
 \end{proof}
  
For $n\ge 1$ let $K_n$ be the $n$-th congruence subgroup in $\GL(N, \OO_F)$. 
\begin{lem}\label{asymptotic_prince} 
$$ \sum_{\chi\in X} (\dim V_{\chi}^{K_{n+1}})^2 \sim (I_{n_0}: K_{n+1}).$$ 
\end{lem}
\begin{proof} 
If $\chi\in X(n)$ then  $K_{n+1}$ acts trivially on $V_{\chi}$ and   $K_n$ does not act trivially on $V_{\chi}$, and, since $V_{\chi}$ is irreducible,  $V_{\chi}^{K_n}=0$. Hence, it is enough to check that 
\begin{equation}\label{growth_n}
 \sum_{\chi\in X(n)} (\dim V_{\chi})^2 \sim (I^{n_0}: K_{n+1}).
 \end{equation}
Let $U^-$ be the subgroup of unipotent lower triangular matrices in $G$. Then 
\begin{equation} \label{dim_V_chi}
 \begin{split} 
 \dim V_{\chi}&= ( I^{n_0}: T^{n_0} J^{n+1})= (I^{n_0}\cap U^-: J^{n+1} \cap U^-)\\
 & = \prod_{i=1}^{N-1}  q_F^{(N-i)(n-n_0-i)}\sim q_F^{n N(N-1)/2}
 \end{split}
 \end{equation}
Since $|X(n)|\sim q_F^{nN}$ by \eqref{size_X_n} the right-hand-side of \eqref{growth_n} grows as $q_F^{n N^2}$. Since
\begin{equation} \label{index}
(I^{n_0}: K_{n+1})= q_F^{(n+1 -n_0) N(N-1)/2} q_F^{(n+1-n_0) N} q_F^{(n+1)N(N-1)/2}\sim q_F^{n N^2},
\end{equation} 
we deduce  that \eqref{growth_n} holds. 
\end{proof}

\begin{prop} The set $\{ V_{\chi}\}_ {\chi\in X}$ captures every projective object in $\Mod^{\pro}_{I^{n_0}}(\OO)$.
\end{prop} 
\begin{proof} The assertion follows from Proposition \ref{criterion} applied with $K=I^{n_0}$ and $K_n$ as in Example \ref{unramified_order} with $L_0= \OO_F^{\oplus N}$, 
		using Lemma \ref{asymptotic_prince}.
\end{proof} 

\subsection{Supercuspidal types for $\GL(N, F)$}\label{sec_super}
Let $F$ be a finite extension of $\Qp$ and let $E$ be a field extension of $F$ of degree $N$. Let $f=f(E|F)$ and $e=e(E|F)$ denote the inertial degree and the ramification 
index of $E$ over $F$ respectively. 

\begin{defi}[{\cite[(1.4.14)]{BK}}] An element $\alpha\in E$ is minimal over $F$ if $e$ is prime to $v_E(\alpha)$ and $\varpi_F^{-v_E(\alpha)} \alpha^e +\pp_E$ generates 
the field extension $k_E/k_F$. 
\end{defi}

 \begin{remar} If $\alpha$ is minimal then $E=F[\alpha]$, since the fields have the same ramification indices and inertial degrees over $F$.
 \end{remar}
 
 \begin{remar} If $\alpha$ is minimal  over $F$ then $\alpha +a$ is also minimal over $F$ for all $a\in E$ with $v_E(a)> v_E(\alpha)$.
 \end{remar}

In general, not every extension $E$ will contain an element which is minimal over~$F$, as the following example shows. 
 
 \begin{examp} Let $f$ be an even integer, $e=p^{f/2}+1$, $F=\Qp$ and let $E$ be the compositum of $\Qp(p^{1/e})$ and the unramified extension 
 of $\Qp$ of degree $f$.  If $x\in (k_E^{\times})^e$ then  $x^{p^{f/2}}=x$, and hence $x$ cannot generate $k_E$ over $k_F$. Any $\alpha\in E$ will 
be of the form $p^{n/e}\xi$ for some $\xi\in \OO_E^{\times}$ and the image of $\varpi^{-n} \alpha^e$ in $k_E$ will be equal to the image of $\xi^e$, 
and hence cannot generate $k_E$ over $k_F$. 
\end{examp}
 
 However, there are plenty extensions that do. For example, let $F'$ be a finite unramified extension of $F$,  
 and let $f(x)=x^e+\ldots + \varpi_F \xi$ be an Eisenstein polynomial in $F'[x]$, such that $\xi +\pp_{F'}$ generates $k_{F'}/k_F$ then $E=F'[x]/(f(x))$ is a 
 field and $x +(f(x))$ is  minimal over $F$. In particular, all unramified or totally ramified extensions will contain a minimal element. 
 
   \begin{lem}\label{many_min} If $\alpha\in E$ is minimal over $F$ then $\alpha^m$ is also minimal over $F$ for all $m$ prime to $e(q_E-1)$. 
 \end{lem} 
 \begin{proof} If $x\in k_E$ then $x$ generates $k_E$ over $k_F$ if and only if it is not contained in any proper subfield, which 
 is equivalent to the order of $x$ in the group $k_E^{\times}$  not dividing $q_F^d-1$ for all divisors $d$ of $f$ with $d<f$. 
 If $m$ is prime to $q_F^f-1$ then $x$ and $x^m$ have the same order, which implies the assertion. 
 \end{proof}
 
 From now on we assume that $E$ contains an element that is minimal over $F$. It follows from Lemma \ref{many_min} that for any $n\ge 0$ there 
 exists $\alpha\in E$ minimal over $F$, such that $v_E(\alpha)< -n$. 
 
 Let $A=\End_F(E)$ and $G=\Aut_F(E)$, so that $G\cong \GL(N, F)$, $E\subset A$ and $E^{\times}\subset G$.  Let $\mathfrak A$ be the hereditary $\OO_F$-order in $A$ associated to 
 the lattice chain $\{\pp_E^i: i\in \ZZ\}$ and let $\Pp$ be its  Jacobson radical. Then $\Pp^n = \{a\in A: a \pp_E^i\subset \pp_E^{i+n}, \forall i\in \ZZ\}$  for all $n\in \ZZ$.
 We let $\UU(\Aa)=\UU^0(\Aa)= \Aa^{\times}$ and  $\UU^n(\Aa)=1+ \Pp^n$ for $n\ge 1$. Then for all $n\ge 0$, $\UU^n(\Aa)$ is a normal subgroup of $\UU(\Aa)$ normalized by $E^{\times}$. 

Let $\alpha\in E$ be minimal with respect to $E/F$ with $v_E(\alpha)=-n<0$. Then $[\Aa, n, 0, \alpha]$ is a simple stratum in the sense of \cite[(1.5.5)]{BK}. For $m\ge 0$, we let 
$$ H^{m}(\alpha):= \UU_E^m \UU^{[n/2]+1}(\Aa), \quad J^m(\alpha):= \UU_E^m \UU^{[(n+1)/2]}(\Aa).$$
Since $\alpha$ is minimal and $[E:F]=N$, these groups coincide with those defined in \cite[(3.1.14)]{BK}. If it is clear that $v_E(\alpha)$ is fixed and is equal to $-n$ then we will 
write $H^{m}$ and $J^m$.

We fix an additive character  $\psi_F: F \rightarrow \Lbar^{\times}$, which is trivial on $\pF$ and  non-trivial on $\OO_F$. If $b\in A$ then $\psi_b: A\rightarrow \Lbar^{\times}$ is the function 
$a\mapsto \psi_F(\tr_A(b(a-1)))$. The restriction of $\psi_{\alpha}$ to $\UU^{[n/2]+1}(\mathfrak A)$ defines  a character, which is trivial on $\UU^{n+1}(\Aa)$ and non-trivial on $\UU^n(\Aa)$. 
Let $\mathcal C(m, \alpha)$ be the set of characters $\theta: H^{m+1}(\alpha)\rightarrow \Lbar^{\times}$, such that the restriction of $\theta$ to $\UU^{[n/2]+1}(\mathfrak A)$ is equal to $\psi_{\alpha}$.
Since $\alpha$ is minimal and $[E:F]=N$, $\mathcal C(m, \alpha)$ is the set of simple characters defined in \cite[(3.2.1)]{BK}. If $\theta\in \mathcal C(m, \alpha)$ then the $G$-intertwining of $\theta$ is 
equal to $E^{\times} J^{m+1}(\alpha)$. It follows  from \cite[(8.3.3)]{BF}, using \cite[(3.4.1)]{BK}, that there is a unique irreducible representation $\eta(\theta)$ of $J^{m+1}(\alpha)$, such that 
$\Hom_{H^{m+1}(\alpha)}(\theta, \eta(\theta))\neq 0$. Thus, 
\begin{equation}\label{handle_eta} 
\Indu{H^{m+1}(\alpha)}{J^{m+1}(\alpha)}{\theta}\cong \eta(\theta)^{\oplus a},
\end{equation}
where $a$ is an integer such that $a \dim \eta(\theta)= (J^{m+1}(\alpha): H^{m+1}(\alpha))$. It is shown in  \cite[(8.3.3)]{BF} that 
\begin{equation}\label{dim_eta}
(\dim \eta(\theta))^2= (J^{m+1}(\alpha): H^{m+1}(\alpha))=(\UU^{[(n+1)/2]}(\Aa): \UU^{[n/2]+1}(\Aa)).
\end{equation}
In particular, $a=\dim \eta(\theta)$ and is equal to $1$ if $n$ is odd, and is equal to $\sqrt{(\Aa: \Pp)}$ if $n$ is even. 

\begin{defi} 
$V(\alpha, \theta):=\Indu{J^{m+1}(\alpha)}{\UU^{m+1}(\Aa)}{\eta(\theta)}$. 
\end{defi}

 \begin{lem} $V(\alpha, \theta)$ is an irreducible representation of $\UU^{m+1}(\Aa)$. 
 \end{lem} 
 \begin{proof} It follows from \eqref{handle_eta} that the $G$-intertwining of $\eta(\theta)$ is equal to the $G$-intertwining of $\theta$. Hence, 
 the $\UU^{m+1}(\Aa)$-intertwining of $\eta(\theta)$ is equal to $\UU^{m+1}(\Aa)\cap E^{\times} J^{m+1}(\alpha)= J^{m+1}(\alpha)$. Hence, 
 $V(\alpha, \theta)$ is irreducible.
 \end{proof}

 \begin{prop}\label{types_super}  If $\pi$ is a smooth irreducible $\Lbar$-representation of $G$ then 
\begin{equation}\label{equiv_super}
\Hom_{\UU^{m+1}(\Aa)}( V(\alpha, \theta), \pi)\neq 0 \iff\pi\cong \cIndu{E^{\times}J^0(\alpha)}{G}{\Lambda},
\end{equation}
where $\Lambda$ is an  irreducible representation, such that $\Lambda|_{H^{m+1}(\alpha)}$ is isomorphic to a direct sum of copies of $\theta$. In particular, 
$\pi$ is supercuspidal. 

Moreover, if $E$ is a tame extension of $F$ and $\rho$ is the $N$-dimensional representation of $W_F$ corresponding to $\pi$ via the 
local Langlands correspondence, then \eqref{equiv_super} is equivalent to $\rho\cong \Indu{W_E}{W_F}{\psi}$, where $\psi: W_E \rightarrow \Lbar^{\times}$ 
is a character, such that if we identify $W_E^{\mathrm{ab}}\cong E^{\times}$ then $\psi|_{\UU^{m+1}_E}=\theta$.
 \end{prop} 
 \begin{proof} Since $\alpha$ is minimal over $F$ and $[E:F]=N$, \cite[(3.2.5), (3.3.18)]{BK} imply that $\Indu{H^{m+1}}{H^1}{\theta}\cong \bigoplus \tilde{\theta}$, 
 where the sum is taken over all $\tilde{\theta}\in \mathcal C(0, \alpha)$, such that $\tilde{\theta}|_{H^{m+1}}= \theta$. It follows from 
 \eqref{handle_eta} that $\Indu{H^1}{J^1}{\tilde{\theta}}\cong \eta(\tilde{\theta})^{\oplus a}$. Theorem (5.2.2) of \cite{BK} implies that 
 $\Indu{J^1}{J^0}{\eta(\tilde{\theta})}\cong \bigoplus_{\sigma} \kappa\otimes \sigma$, where $\kappa$ is an irreducible representation of $J^0$, such that $\kappa|_{J^1}\cong \eta(\tilde{\theta})$, 
 the so called $\beta$-extension, and the sum is taken over all characters $\sigma$ of $\UU_E/\UU^1_E$, which we inflate to $J^0$ via the isomorphism 
 $J^0/J^1\cong \UU_E/\UU^1_E$. Thus 
 \begin{equation} 
 \Indu{H^{m+1}}{J^{0}}{\theta}\cong (\bigoplus \lambda)^{\oplus a}
 \end{equation}
  where the sum is taken over all simple types $\lambda$, \cite[(5.5.10)]{BK}, associated to a simple stratum $[\mathfrak A, n, 0,\alpha]$, such that 
  $\lambda |_{H^{m+1}}$ is a direct sum of copies of $\theta$. If $\Hom_{J^0}(\lambda, \pi)\neq 0$ then $\pi$ is supercuspidal and there is a unique 
  irreducible representation $\Lambda$ of $E^{\times} J^0$, such that $\Lambda|_{J^0}\cong \lambda$ and $\pi\cong \cIndu{E^{\times}J^0}{G}{\Lambda}$, \cite[(6.2.2)]{BK}.
The first part of the Proposition now follows from \eqref{handle_eta} and Frobenius reciprocity. The second part follows from \cite{BH}. 
 \end{proof} 
 
 \begin{lem}\label{invariants_super} $\UU^{n+1}(\Aa)$ acts trivially on $V(\alpha, \theta)$. The space of $\UU^n(\Aa)$-invariants in $V(\alpha, \theta)$ is zero. 
 \end{lem}
 \begin{proof} This follows from the fact that the restriction of $V(\alpha, \theta)$ to $\UU^{[n/2]+1}(\Aa)$ is isomorphic to a direct sum of copies of $\psi_{\alpha}$. 
 \end{proof}
 
 \begin{lem}\label{dim_V}
$$ \dim V(\alpha, \theta)=\frac{( \UU^{m+1}(\Aa): \UU^{[n/2]+1}(\Aa))}{ (\UU^{m+1}_E : \UU^{[n/2]+1}_E)\dim \eta(\theta)}\sim (\UU^{m+1}(\Aa): \UU^{n+1}(\Aa))^{1/2} q_E^{-n/2}.$$
 \end{lem}
 \begin{proof} It follows from the definition of $V(\alpha, \theta)$ as induced representation that its dimension is equal to $(\UU^{m+1}(\Aa): J^{m+1}(\alpha)) \dim \eta(\theta)$. The Lemma follows from this using \eqref{dim_eta}.
 \end{proof} 
 
We fix an additive character $\psi_E: E\rightarrow \Lbar^{\times}$, which is trivial on $\pp_E$ and nontrivial on $\OO_E$. There exists a unique map $s: A\rightarrow E$ such that 
 \begin{equation}\label{tame_cor}
 \psi_A(ab)= \psi_E(s(a) b), \quad \forall a\in A, \quad \forall b\in E.
 \end{equation} 
 Moreover, $s$ is an $E\times E$-bimodule homomorphism, such that $s(\Pp^m)= \pp_E^m$ for all $m\in \ZZ$,  \cite[(1.3.4)]{BK}. Since 
 $E$ is maximal in $A$ and $E$ is separable over $F$, $\tr_{A/F}(E)=\tr_{E/F}(E)=F$. It follows from \eqref{tame_cor} that $s(1)\neq 0$. 
  We let $\delta:= v_E(s(1))$.

 \begin{lem}\label{main_lemma} Let $n$, $r$ be natural numbers with $n-\delta> r$.
 Let $\alpha_1,\alpha_2\in E$ be minimal over $F$ with $v_E(\alpha_1)=v_E(\alpha_2)=-n.$ If $\alpha_1 - \alpha_2 \in \pp_E^{-n +\delta+1}$ 
 and $\alpha_1+ \Pp^{-r} = x \alpha_2 x^{-1} + \Pp^{-r}$ for some $x\in \UU(\Aa)$ then $\alpha_1-\alpha_2 \in \pp_E^{-r-\delta}$. 
 \end{lem}
 \begin{proof} Since $\pp^{-n+\delta+1}_E\subset \Pp^{-n+\delta+1}$ we have 
 $$\alpha_2+ \Pp^{-n+\delta+1} = \alpha_1 + \Pp^{-n+\delta+1}=x \alpha_2 x^{-1}+ \Pp^{-n+\delta+1}.$$
 Since $\alpha_2$ is minimal it follows from \cite[Thm.\,2.4]{KM} that $x= u y$ with $u\in E^{\times}$ and $y\in \UU^{\delta+1}(\Aa)$. Then 
 $x \alpha_2 x^{-1}= y\alpha_2 y^{-1}$ and so 
 $ \alpha_1+ \Pp^{-r}= y \alpha_2 y^{-1} +\Pp^{-r}$.  Thus  $\alpha_1 y - y \alpha_2 \in \Pp^{-r}$. By applying $s$, we get 
 $s(y)( \alpha_1 - \alpha_2)\in s(\Pp^{-r})= \pp_E^{-r}$. Since $y\in \UU^{\delta+1}(\Aa)$, $s(y)\in s(1)+s(\Pp^{\delta+1})= s(1)+ \pp_E^{\delta+1}$, hence 
  $v_E(s(y))= \delta$ and $\alpha_1 - \alpha_2\in \pp_E^{-r-\delta}$. 
  \end{proof}
 
 \begin{lem}\label{notinter} For $i=1, 2$ let $\alpha_i\in E$ be minimal over $F$ with $v_E(\alpha_i)=-n$ and $n> 2\delta$, and let $\theta_i\in \mathcal C(m, \alpha_i)$. 
 If $\alpha_1- \alpha_2\in \pp_E^{-n+\delta +1}$ and $\alpha_1-\alpha_2 \not\in \pp_E^{-[n/2]-\delta}$ then $\theta_1$ and $\theta_2$ do not intertwine in $G$. 
 \end{lem}
 \begin{proof} If $g\in G$ intertwines $\theta_1$ and $\theta_2$ then there exists $x\in \UU(\Aa)$ normalizing $H^{m+1}$ such that $\theta_1^x= \theta_2$, \cite[(3.5.11)]{BK}.
 By considering the restrictions of $\theta_1$ and $\theta_2$ to $\UU^{[n/2]+1}(\Aa)$ we obtain that $\psi_{\alpha_1}= \psi^x_{\alpha_2}= \psi_{x \alpha_2 x^{-1}}$ as 
 characters of $\UU^{[n/2]+1}(\Aa)$. Hence, $\psi_A((\alpha_1 - x \alpha_2 x^{-1}) a)=1$ for all $a\in \Pp^{[n/2]+1}$, which implies $\alpha_1- x \alpha_2 x^{-1}\in \Pp^{-[n/2]}$, 
 \cite[p.22]{BK}. Lemma \ref{main_lemma} implies that $\alpha_1-\alpha_2 \in \pp_E^{-[n/2]-\delta}$, leading to a contradiction.
 \end{proof}
 
 \begin{lem}\label{single_alpha} Let $\alpha\in E$ be minimal with $v_E(\alpha)=-n$, where  $n\ge  2m\ge 0$. Then $|\mathcal C(m, \alpha)| = (\UU^{m+1}_E : \UU^{[n/2]+1}_E)\sim q_E^{n/2}$. 
 Moreover, distinct $\theta_1, \theta_2\in \mathcal C(m, \alpha)$ do not intertwine.
 \end{lem}
 \begin{proof} The set $\mathcal C(m, \alpha)$ is non-empty, \cite[(3.3.18)]{BK}. If $\theta\in \mathcal C(m, \alpha)$ then all the other characters 
 in $\mathcal C(m, \alpha)$ are of the form $\theta \mu$, where $\mu$ is any character of $\UU^{m+1}_E/ \UU^{[n/2]+1}_E$ inflated to $H^{m+1}$.
 This implies the first assertion. The second assertion follows from the proof of \cite[Prop.\,7.3]{thesis}.  
 Namely, if $g$ intertwines $\theta$ and $\theta \mu$ then it will intertwine  their restrictions to $\UU^{[n/2]+1}(\Aa)$, which are both equal to 
 $\psi_{\alpha}$. It follows from \cite[(3.3.2)]{BK} that $g\in E^{\times} \UU^{[(n+1)/2]}(\Aa)$. Any such  $g$ will normalize $H^{m+1}$, and so
 the intertwining maybe rewritten as 
\begin{equation}\label{intertwine} 
 \theta\mu(h)= \theta(g h g^{-1}), \quad \forall h \in H^{m+1}. 
 \end{equation}
 By applying  \cite[(3.3.2)]{BK}  again, we get that $g$ will also intertwine $\theta$ with itself. It then follows from \eqref{intertwine} that $\mu(h)=1$ 
 for all $h\in H^{m+1}$.
 \end{proof}
 
We fix $m\ge 0$.  For $n\ge 2m$, let $I(n)$ be the set of isomorphism classes of the representations of the form $V(\alpha, \theta)$, where 
$\theta\in \mathcal C(m, \alpha)$ with $\alpha\in E$ minimal over $F$ and $v_E(\alpha)=-n$. 
 
 \begin{prop}\label{size_I(n)} $|I(n)|\sim q_E^{n}$.
 \end{prop} 
 \begin{proof} Let us note that it follows from Lemma \ref{many_min}  that there exists $\alpha\in E$ minimal over $F$ with $n:=-v_E(\alpha)$ arbitrarily large. 
It is enough to show that $g(n)\le | I(n)| \le f(n)$, with $f(n)\sim g(n) \sim q_E^n$. For the upper bound we observe that 
$$ | I(n)| \le |\bigcup_{\alpha} \mathcal C(m, \alpha)| \le |\mathcal C(m, \alpha)| | \pp_E^{-n}/ \pp_E^{-[n/2]}| \sim q_E^n,$$
where the union is taken over all minimal $\alpha\in E$ with $v_E(\alpha)=-n$. In the above estimates we use Lemma \ref{single_alpha} and 
the fact that if $\alpha_1-\alpha_2\in \pp_E^{-[n/2]}$ then $\psi_{\alpha_1}=\psi_{\alpha_2}$ on $\UU^{[n/2]+1}(\Aa)$ and hence $\mathcal C(m, \alpha_1)=\mathcal C(m, \alpha_2)$.
For the lower bound we choose a set of representatives $\{a_i\}_i$ of $\pp_E^{-n+\delta+1}/\pp_E^{-[n/2]-\delta}$ and let $\alpha_i:= \alpha+a_i$. Then each $\alpha_i$ is minimal over $F$, 
$\alpha_i - \alpha_j \in \pp_E^{-n +\delta+1}$ and $\alpha_i -\alpha_j= a_i-a_j\not \in \pp_E^{-[n/2]-\delta}$ if $i\neq j$. Lemmas \ref{single_alpha} and \ref{notinter} imply that 
distinct $\theta_1, \theta_2\in \bigcup_i \mathcal C(m, \alpha_i)$ do not intertwine. This implies that $\Hom_{\UU^{m+1}(\Aa)}(\Indu{H^{m+1}}{\UU^{m+1}(\Aa)}{\theta_1}, 
\Indu{H^{m+1}}{\UU^{m+1}(\Aa)}{\theta_2})=0$, and \eqref{handle_eta} implies that 
$V(\alpha, \theta_1)\not\cong V(\alpha', \theta_2)$. Hence 
$$|I(n)|\ge |\bigcup_i \mathcal C(m, \alpha_i)|= |\mathcal C(m, \alpha)| | \pp_E^{-n+\delta+1}/\pp_E^{-[n/2]-\delta}|\sim q_E^n,$$
where we use Lemma \ref{single_alpha} again.
 \end{proof}

 We let  $I=\cup_{n\ge 2m} I(n)$ and for each $i\in I$ we fix a representative $V_i$ of the isomorphism class $i$. 
 
 \begin{prop}\label{capture_super} Let $m$ be large enough, so that $\UU^{m+1}(\Aa)$ is torsion free. Then the set $\{V_i\}_{i\in I}$ captures every projective object in $\Mod^{\pro}_{\UU^{m+1}(\Aa)}(\OO)$. 
 \end{prop}
 \begin{proof} We will deduce the assertion from Proposition \ref{criterion} with $K=\UU^{m+1}(\Aa)$ and $K_n=\UU^{n+1}(\Aa)$. Proposition \ref{uniform} and Lemma \ref{p_power} imply that 
 $\{K_n\}_{n\ge m}$ satisfy the conditions of Proposition \ref{criterion}.  Lemma \ref{invariants_super} implies that $i\in I(n)$ if and only if $\UU^{n+1}(\Aa)$ acts trivially on $V_i$ and 
 $\UU^n(\Aa)$ acts nontrivially on $V_i$. Thus it is enough to show that 
 $$\sum_{i\in I(n)} (\dim V_i)^2 \sim (\UU^{m+1}(\Aa): \UU^{n+1}(\Aa)).$$
 This follows from Lemma \ref{dim_V} and Proposition \ref{size_I(n)}.
 \end{proof}

 \section{Capture a la Scholze}\label{sec_scholze}
 
 The following section is motivated by the proof of Corollary 7.3 in \cite{scholze}. Let $K$ be a pro-finite group and let $\{K_n\}_{n\ge 1}$ be a basis of open neighbourhood of
 $1$ consisting of normal subgroups of $K$. We assume that $K_n$ is pro-$p$ for all $n\ge 1$. Let 
 $\psi_n: K_n \rightarrow \overline{L}^*$ be a smooth character for all $n\ge 1$. 
 
 Let $R$ be a complete noetherian local $\OO$-algebra with residue field $k$. Let $M$ be a
 finitely generated $R\br{K}$-module. 
 
 \begin{prop}\label{easy_does_it} Assume $M$ is projective as $\OO\br{K}$-module. If  $\phi\in\End^{\cont}_{R\br{K}}(M)$ kills $\Hom_{\OO\br{K_n}}^{\cont}(M, \psi_n^*)$ for all $n\ge 1$ then 
 $\phi$ kills $M$.
 \end{prop}
 \begin{proof} Since $\psi_n$ is smooth, $\psi_n(K_n)$ is a finite subgroup of $\overline{L}^*$. Since $K_n$ is pro-$p$, the order of $\psi_n(K_n)$ is a power of $p$. Hence, we may assume that $\psi_n$ takes values in $\OO_n^{\times}$, where $\OO_n$ is 
 the ring of integers of a finite extension $L_n$ of $L$ obtained by adding some $p$-power roots of unity to $L$.  Let $\OO_n(\psi_n^*)$ be a free $\OO_n$-module of rank $1$, on which $K_n$ acts by $\psi_n^*$. Since $\OO_n(\psi_n^*)$ is an $\OO\br{K_n}$-submodule of $\psi_n^*$ we have an injection
  $$\Hom_{\OO\br{K_n}}^{\cont}(M, \OO_n(\psi_n^*))\hookrightarrow 
 \Hom_{\OO\br{K_n}}^{\cont}(M, \psi_n^*).$$
 Moreover, this is an injection of $\End_{\OO\br{K}}^{\cont}(M)$-modules. Thus $\phi$ annihilates 
 $$\Hom_{\OO\br{K_n}}^{\cont}(M, \OO_n(\psi_n^*))$$
  for all $n\ge 1$. Let $\varpi_n$ be a uniformizer in $\OO_n$. Then the residue field is equal to $k$ and all the $p$-power roots of unity in $\OO_n$ are mapped to $1$ in $k$.  We have an exact sequence of $K_n$-representations 
\begin{equation}\label{seq}
0\rightarrow \OO_n(\psi_n^*)\overset{\varpi_n}{\longrightarrow} \OO_n(\psi_n^*)\rightarrow k\rightarrow 0
\end{equation}
 with the trivial action of $K_n$ on the quotient. Since $M$ is projective as an $\OO\br{K}$-module, 
 it is also projective as an $\OO\br{K_n}$-module and 
 by applying $\Hom^{\cont}_{\OO\br{K_n}}(M, \ast)$ to \eqref{seq} we see that 
 $\Hom^{\cont}_{\OO\br{K_n}}(M, k)$ is a quotient of $\Hom_{\OO\br{K_n}}^{\cont}(M, \OO_n(\psi_n^*))$ and thus is  killed by $\phi$. This module is isomorphic to the continuous $k$-linear dual of $K_n$-coinvariants 
 of $M/\varpi M$, which we denote by $(M/\varpi M)_{K_n}$. Using Pontryagin duality we conclude that $\phi$ acts trivially on $(M/\varpi M)_{K_n}$.  
 
 Let $\mm$ be the maximal ideal of $R$ and let $I_{m,n}= \mm^m R\br{K} + I(K_n)R\br{K}$, where $I(K_n)$ is 
 the augmentation ideal of $\OO\br{K_n}$. Since $M$ is finitely generated as $R\br{K}$-module, its topology coincides with the topology defined by the ideals $I_{m,n}$, see \cite[(5.2.17)]{nsw2}. Thus 
 $$M/\varpi M \cong \varprojlim_{n, m}  (M/\varpi M)/ I_{n,m} (M/\varpi M). $$
 Since $(M/\varpi M)/ I(K_n) ( M/\varpi M)= (M/\varpi M)_{K_n}$, by the previous part we deduce that $\phi$ kills all the terms in this projective limit. Thus $\phi$ kills $M/\varpi  M$ and so $\phi(M)$ is contained in $\varpi M$.
 
  If $\phi$ is non-zero then there is a largest integer $n$ such that $\phi(M)$ is contained in $\varpi^n M$ and 
  not contained in 
 $\varpi^{n+1} M$. We may replace $\phi$ by $\frac{1}{\varpi^n}\phi $ in the argument above to obtain a contradiction. 
 \end{proof}
 
 \begin{cor}\label{easy} Let 
$M$ be as in Proposition \ref{easy_does_it} and assume that $R$ acts faithfully on $M$. Let $\{ V_i\}_{i\in I}$ be the family of irreducible subquotients of $\Ind_{K_n}^K \psi_n$ for all $n\ge 1$.  For each $i\in I$ let $\mathfrak a_i$ be the $R$-annihilator of 
$\Hom_{\OO\br{K}}^{\cont}(M, V_i^*)$.  Then $\cap_{i\in I} \mathfrak a _i=0$. 
\end{cor}
\begin{proof} This follows from Proposition \ref{easy_does_it} by observing that $\Ind_{K_n}^K \psi_n$ is a finite direct 
sum of some $V_i$'s with some multiplicity and then using Frobenius reciprocity to obtain an isomorphism of $R$-modules
$$\Hom^{\cont}_{\OO\br{K_n}}(M, \psi_n^*)\cong \Hom_{\OO\br{K}}^{\cont}(M, (\Ind_{K_n}^K \psi_n)^*).$$
\end{proof}

\begin{remar} Although the theory outlined in  section \ref{generalities} is more general, 
all the smooth examples in section \ref{examples} are covered by Corollary \ref{easy}, 
which avoids all the intricate counting arguments. For example in the supercuspidal case, 
considered in section \ref{sec_super}, it is enough to apply Corollary \ref{easy} with $K=\UU(\Aa)$ and 
consider families  $(\UU^{[n/2]+1}(\Aa), \psi_{\alpha})$, for all $n\ge 1$, such that $[\Aa, n, 0, \alpha]$ 
is a simple stratum, $E=F[\alpha]$ and $\alpha$ is minimal over $F$.
\end{remar}

\section{Applications}\label{sec_appl}

\subsection{Global applications}\label{global}

Let $\mathbb G$ be a reductive group over $\mathbb Q$ with the property
that the maximal $\mathbb Q$-split torus in the centre
of $\mathbb G$ is a maximal $\mathbb R$-split torus in $\mathbb G$.
This means that $\mathbb G(\mathbb R)$ is compact modulo its centre,
and that if the level $K_f$ is small enough, then
$\mathbb G(\mathbb Q)$ acts %without fixed points
with trivial stabilizers
on $\mathbb G(\mathbb A)/A_{\infty}^{\circ} K_{\infty}^{\circ} K_f.$ We 
let $$ Y(K_f):= \mathbb G(\mathbb Q)\backslash\mathbb G(\mathbb A)/A_{\infty}^{\circ} K_{\infty}^{\circ} K_f.$$
Here we use the usual notation that $A_{\infty}$ is the group
of real points of the maximal $\mathbb Q$-split torus in the
centre of $\mathbb G$, and $K_{\infty}$ denotes
a choice of maximal compact subgroup of $\mathbb G(\mathbb R)$;
and, of course, the superscript ${}^{\circ}$ denotes the connected
component of the identity.

We are in the situation of \cite[(3.2)]{em0}.
In particular, we may consider the completed cohomology
$\widetilde{H}^0(K^p)$ with $\mathcal O$-coefficients,
for some choice of tame level $K^p$;  
recall that this is defined as 
$$\widetilde{H}^0(K^p) := \varprojlim_s \varinjlim_{K_p}
H^0( Y(K_pK^p) , \mathcal O/ \varpi^s),$$
where $K_p$ runs over all compact open subgroups of $\mathbb G(\Q_p)$.
We may also consider the completed {\em homology}, with $\mathcal O$-coefficients,
which we denote by $M$;
it is defined as
$$M := \varprojlim_{K_p} 
H_0(Y(K_pK^p) ,\mathcal O).$$
We then have 
that $M$ is finitely generated and (if $K^p$ is small enough)
free over~$\mathcal O\br{K}$, and
\begin{equation}\label{ref_wants}
\widetilde{H}^0(K^p) = \Hom^{\cont}_{\mathcal O}(M,\OO),
\end{equation}
where $K$ is a compact open subgroup of $\mathbb G(\Qp)$.

For each $K_p$, let $\mathbb T(K_p)$ denote the commutative algebra
of endomorphisms of
$H_0
(Y(K_pK^p) ,
\mathcal O)$
generated by the spherical Hecke operators
at the primes $\ell \neq p$ at which $K^p$ is unramified (i.e.\ at which
$\mathbb G$ is unramified and $K^p$ is of the form
$K^{p,\ell} K_{\ell}$, where $K_{\ell}$ is maximal compact hyperspecial).
We then define
$$\mathbb T :=  \varprojlim_{K_p} \mathbb T(K_p);$$
this is an algebra of $G$-equivariant endomorphisms of $M$.
We endow each $\mathbb T(K_p)$ with its $\varpi$-adic topology,
and endow $\mathbb T$ with its projective limit topology.
Then $\mathbb T$ is a pro-Artinian $\mathcal O$-algebra
(indeed, since each $\mathbb T(K_p)$ is finitely generated as an $\mathcal O$-modules,
we may describe $\mathbb T$ as
the projective limit of (literally) finite $\mathcal O$-algebras),
and is the product (as a topological ring) of finitely many local pro-Artinian
$\mathcal O$-algebras,
these local factors being in bijection with the maximal ideals of~$\mathbb T$;
we let $\mathbb T_{\mathfrak m}$ denote the factor corresponding to a given
maximal ideal~$\mathfrak m$.
The local factor $\mathbb T_{\mathfrak m}$ is naturally identified with
the $\mathfrak m$-adic completion of $\mathbb T$ as an abstract ring; 
note, though, that the topology on $\mathbb T_{\mathfrak m}$ coincides with the
$\mathfrak m$-adic topology if and only if $\mathbb T_{\mathfrak m}$
is Noetherian.   In fact, because of the anticipated relationship
with Galois deformation rings, we expect that $\mathbb T$,
and hence each $\mathbb T_{\mathfrak m}$, {\em is} Noetherian,
but as far as we know, this is not proved in the generality that we consider here.
If $V$ is any locally algebraic representation of~$K$,
then it follows from \eqref{ref_wants} and Schikhof duality \cite[Thm.\,3.5]{iw} that
$$\Hom_{\mathcal O\br{K}}^{\cont}(M,V^*) =
\Hom_{ K} \bigl(V, \widetilde{H}^0(K^p)\bigr),$$
and the target is isomorphic to the space of algebraic automorphic forms
on $\mathbb G$ with values in $V^*$ \cite[Prop.\,3.2.4]{em0}.
In particular, the action of $\mathbb T \otimes_{\mathcal O}L$ on
this space is semi-simple.

Now, let $\mathfrak m$ be a maximal ideal of $\mathbb T$.
We may form the localization $M_{\mathfrak m}$; this is a direct
summand of $M$, and so projective as an $\mathcal O\br{K}$-module.
It has a commuting action of the local ring $\mathbb T_{\mathfrak m}$.
The space $\Hom_{\mathcal O\br{K}}^{\cont}(M_{\mathfrak m},V^*)$
is  a finite-dimensional $L$-vector space.
The action of $\mathbb T_{\mathfrak m}$ on this space
factors through an order in the finite-dimensional $L$-algebra
generated by the Hecke operators.

Remark~\ref{rem:semi-simple} applies, and so Proposition~\ref{dense_cap} together with Remark \ref{nonnoetherian}
shows that if the set $\{V_i\}_{i \in I}$ captures $M$,
then the systems of Hecke eigenvalues arising
in the spaces of algebraic automorphic forms
$\Hom_{\mathcal O\br{K}}^{\cont}(M,V_i^*)$ 
are Zariski dense in $\Spec \mathbb T_{\mathfrak m}$.

\subsection{Hecke algebras and functoriality} Let us suppose that we have two groups $\mathbb G_1$ and $\mathbb G_2$ as in the previous 
subsection. We keep the notation of the previous section, 
except that we add the index $1$ or $2$ to indicate with respect to which group the objects are defined. Assume that both Hecke algebras $\mathbb T_1$ and $\mathbb T_2$ are completions of the quotients 
of the same universal Hecke algebra $\mathbb T^{\univ}$, which is a polynomial ring over 
$\OO$ with variables corresponding to the Hecke operators at the unramified places for both groups.
Let $\mm$ be a maximal ideal of $\mathbb T^{\univ}$ and assume that $\mm$ is a maximal ideal of both $\mathbb T_1$ and $\mathbb T_2$. Then both $\mathbb T_{1,\mm}$ and $\mathbb T_{2,\mm}$ are quotients of the $\mm$-adic completion of $\mathbb T^{\univ}$, which we denote by $\widehat{\mathbb T}^{\univ}_{\mm}$.

Let $\{V_i\}_{i\in I}$ be a family of irreducible  locally algebraic representations of $K_1$, which captures $M_1$. 
As explained above, a maximal ideal $x\in \mSpec \widehat{\mathbb T}^{\univ}_{\mm}[1/p]$ in the
support of $\Hom_{ K_1} \bigl(V_i, \widetilde{H}^0(K_1^p)_{\mm}\bigr)$ corresponds to an algebraic
automorphic form on $\mathbb G_1$. Let us further assume that functoriality between 
$\mathbb G_1$ and $\mathbb G_2$ is available and we may transfer such automorphic forms 
on $\mathbb G_1$ to automorphic forms on $\mathbb G_2$. More precisely, we assume that 
for such $x\in \mSpec \widehat{\mathbb T}^{\univ}_{\mm}[1/p]$, the map $ \widehat{\mathbb T}^{\univ}_{\mm}\rightarrow \kappa(x)$, where $\kappa(x)$ is the residue field of $x$, fits into a commutative diagram:
\begin{displaymath}
\xymatrix{ \widehat{\mathbb T}^{\univ}_{\mm}\ar@{>>}[r]\ar@{>>}[d]& \mathbb T_{1,\mm}\ar[d]\\
\mathbb T_{2,\mm} \ar[r] &\kappa(x).}
\end{displaymath}
Let $\Sigma$ be the subset of $\mSpec \widehat{\mathbb T}^{\univ}_{\mm}[1/p]$ consisting of maximal ideals contained in the support of $\Hom_{ K_1} \bigl(V_i, \widetilde{H}^0(K_1^p)_{\mm}\bigr)$
for some $i\in I$. Since $\{V_i\}_{i\in I}$ captures $M_1$ it will also capture its direct summand 
$M_{1, \mm}$. Hence, the image of $\widehat{\mathbb T}^{\univ}_{\mm} \rightarrow \prod_{x\in \Sigma} \kappa(x)$ is equal to $\mathbb T_{1, \mm}$. The commutativity of the above diagram implies that this map factors through $\mathbb T_{2, \mm}$. Thus we obtain a surjection 
$\mathbb T_{2, \mm}\twoheadrightarrow \mathbb T_{1, \mm}$ and so the Hecke algebra $\mathbb T_{2, \mm}$ acts on the completed cohomology $\widetilde{H}^0(K^p_1)$.

The meta-argument outlined above can be modified to give a different proof of Corollary 7.3 in 
\cite{scholze}, where the functoriality is given by the (classical) Jacquet--Langlands correspondence, by 
taking $\{V_i\}_{i\in I}$ to be a family of supercuspidal types considered in subsection \ref{sec_super} and 
using Proposition \ref{capture_super} together with Proposition \ref{types_super}. A similar example has 
been  discussed in \cite[ \S 3.3.2]{matt_icm}.

\subsection{The case of definite unitary groups}
In the case when $\mathbb G = GU(n)$,
we let $\mathbb T_{\mathfrak m}'$ denote the subring
of $\mathbb T_{\mathfrak m}$
topologically generated by spherical Hecke operators
at unramified primes $\ell$ which are furthermore split
in the appropriate quadratic imaginary field.
The ring $\mathbb T_{\mathfrak m}'$ is a subring
of $\mathbb T_{\mathfrak m}$,
and (assuming that $\mathfrak m$ corresponds to an irreducible
$\overline{\rho}$) is a quotient
of a global deformation ring $R_{\overline{\rho}}$.  
(See \cite[\S 2.4]{six-authors} for an indication of what exact
kinds of Galois representations $R_{\overline{\rho}}$ parameterizes.)

Thus we get the Zariski density in $\Spec \mathbb T_{\mathfrak m}'$
of the collection of Galois representations arising from the various spaces
$\Hom_{\mathcal O\br{K}}^{\cont}(M,V_i^*).$

We formulate this discussion into a theorem to make it more concrete. We freely use the notation of \cite{six-authors}. 
 Let $R$ be the Hecke algebra denoted by 
$\TT^{S_p}_{\xi, \tau}(U^{\pp}, \OO)_{\mm}$ in \cite{six-authors} before Corollary 2.11 and let $M= \tilde{S}_{\xi, \tau}(U^\pp, \OO)^d_\mm$. 
It follows from the proof of \cite[Prop.\,2.10]{six-authors}  that $\tilde{S}_{\xi, \tau}(U^\pp, \OO)^d_\mm$ is a finitely generated projective 
$\OO\br{\GL_n(\OO_F)}$-module, where $\pp$ is a fixed place above $p$ of totally real field denoted by $\widetilde{F}^+$ in \cite{six-authors} and 
$F$ is a finite extension of $\Qp$ such that $F$ is the $\pp$-adic completion of  $\widetilde{F}^+$.

 Let $\Sigma^{\mathrm{cl}}$ be the subset of $\mSpec R[1/p]$ corresponding to the Hecke eigenvalues corresponding 
to the classical (algebraic) automorphic forms on the unitary group considered in \cite[\S 2.3]{six-authors}. To each $x\in \Sigma^{\mathrm{cl}}$
one may associate a Galois representation, which is potentially semi-stable at $\pp$. For each embedding $\kappa: F\rightarrow \Qpbar$ we fix 
an $n$-tuple of distinct integers $k_{\kappa}:=(k_{\kappa, 1}> k_{\kappa, 2} > \ldots > k_{\kappa, n})$ and let $\underline{k}$ be the multiset 
$\{k_\kappa\}_{\kappa: F \hookrightarrow \Qpbar}$.  Let $\Sigma(\underline{k})$ be the subset of $\Sigma^{\mathrm{cl}}$
 corresponding to the automorphic forms such that the restriction of the associated Galois representation at the decomposition group at $\pp$ 
 has Hodge--Tate weights equal to $\underline{k}$. 
 
 To each classical automorphic form we may associate an automorphic representation.  Let $\Sigma_{ps}$ be the subset of $\Sigma^{\mathrm{cl}}$ such that 
 the local factor of the automorphic representation at $\pp$ is a principal series representation of $\GL_n(F)$. In terms of Galois representations
 the subset $\Sigma_{ps}$ correspond to those automorphic forms such that the restriction of the associated Galois representation at the decomposition group at $\pp$ becomes crystalline after a restriction to the Galois group of a finite abelian extension of $F$.

 We fix either a totally ramified or unramified extension $E$ of $F$ of degree $n$ and identify $\GL_n(F)$ with $\Aut_F(E)$. More generally we could take 
 any extension $E$ of $F$ of degree $n$, which is generated by a minimal element in the sense of section \ref{sec_super}. Let $\Sigma_{sc}$ be the subset 
 of $\Sigma^{\mathrm{cl}}$ such that  the local factor of the automorphic representation at $\pp$ is a supercuspidal representation of $\GL_n(F)$ containing a 
 minimal stratum $[\Aa, m, 0, \alpha]$ with $E= F[\alpha]$. If $E$ over $F$ is tamely ramified and $\rho$ is a Galois representation associated to an automorphic form 
 in $\Sigma_{sc}$ then the Weil--Deligne representation associated to 
 the restriction of $\rho$ to the decomposition group at $\pp$ is isomorphic to the induction from $W_E$ to $W_F$ of a $1$-dimensional representation; the 
 requirement for $[\Aa, m, 0, \alpha]$ to be minimal 
  imposes further restrictions on the $1$-dimensional representation of $W_E$.

 \begin{thm}\label{dense_hecke} The sets $\Sigma(\underline{k})\cap \Sigma_{ps}$ and $\Sigma(\underline{k})\cap \Sigma_{sc}$ are both Zariski dense inside  $\Spec \TT^{S_p}_{\xi, \tau}(U^{\pp}, \OO)_{\mm}$.
 \end{thm}
 
 \begin{proof} We will carry out the proof in the supercuspidal case. The proof in the principal series case is the same using the results of section \ref{sec_prince}. Out of the multiset $\underline{k}$ one may manufacture a highest weight and hence a corresponding irreducible algebraic representation 
 of $\mathbb G':=\Res^F_{\Qp} \GL_n/F$, see section \cite[\S 1.8]{six-authors}. Let $W$ be its evaluation at $L$ and restriction to $\GL_n(F)$ via 
 $\GL_n(F)=\mathbb{G}'(\Qp)\hookrightarrow \mathbb{G}'(L)$.  
 
 We will now apply the theory developed in section \ref{sec_super} with 
 $R=  \TT^{S_p}_{\xi, \tau}(U^{\pp}, \OO)_{\mm}$ and 
 $M= S(U^\pp, \OO)_{\mm}^d$. Let $K:=\UU^{m+1}(\Aa)$. 
 Since $M$ is projective as $\OO\br{\GL_n(\OO_F)}$-module 
 it is also projective as $\OO\br{K}$-module. 
 It follows from  Lemma \ref{weight} and Proposition \ref{capture_super} that the family $\{W\otimes V(\alpha, \theta)\}_{\alpha, \theta}$, where $\alpha$ runs over the elements of $E$, such that 
 $E=F[\alpha]$ such that $[\Aa, n', 0, \alpha]$ is a simple stratum with $n'\ge 2m$ and $\alpha$ is minimal over $F$ and $\theta$ runs over all simple characters
 $\theta\in \mathcal C(m, \alpha)$, captures~$M$. 
 Since the Hecke algebra acts faithfully on the module by construction 
 we only have to understand the
 $R$-annihilator of $\Hom^{\cont}_{\OO\br{K}}(M, (W\otimes V(\alpha, \theta))^*)$ and then the assertion follows from Proposition \ref{dense_cap}. We first note that 
 since $M$ is finitely generated over $\OO\br{K}$ and both $W$ and $V(\alpha, \theta)$ are finite dimensional $L$-vector space $\Hom^{\cont}_{\OO\br{K}}(M, (W\otimes V(\alpha, \theta))^*)$ is a finite dimensional $L$-vector space. By Schikhof 
 duality we get that 
$$\Hom^{\cont}_{\OO\br{K}}(M, (W\otimes V(\alpha, \theta))^*)\cong 
\Hom_K(W\otimes V(\alpha, \theta), \tilde{S}(U^{\pp}, \OO)_{\mm}\otimes_{\OO} L),$$
 where $\tilde{S}(U^{\pp}, \OO)_{\mm}\otimes_{\OO} L\cong \Hom^{\cont}_{\OO} (M, L)$ is an admissible unitary $L$-Banach space representation of $\GL_n(F)$. Since 
 $W\otimes V(\alpha, \theta)$ is a locally algebraic representation of $K$, the image of a $K$-invariant homomorphism will be contained in the subspace of locally algebraic vectors $(\tilde{S}(U^{\pp}, \OO)_{\mm}\otimes_{\OO} L)^\alg$. It follows from \cite[Prop.\,3.2.4]{em0} that the representation $(\tilde{S}(U^{\pp}, \OO)_{\mm}\otimes_{\OO} L)^\alg$ can be computed in terms of classical automorphic forms. 
 In particular, it follows from \cite[Prop.\,6.11]{gross} that the action of 
 $R[1/p]$ on 
 $(\tilde{S}(U^{\pp}, \OO)_{\mm}\otimes_{\OO} L)^\alg$ is semi-simple, and (essentially by definition) the support is equal to $\Sigma^{\mathrm{cl}}$. 
 Moreover, if $x\in \Sigma^{\mathrm{cl}}$ and $\mm_x$ is the corresponding maximal ideal of $R[1/p]$ then the subspace of $(\tilde{S}(U^{\pp}, \OO)_{\mm}\otimes_{\OO} L)^\alg$ annihilated by $\mm_x$ is isomorphic to $\pi_{x,\pp} \otimes W_x$, where 
 $\pi_{x,\pp}$ is the local component at $\pp$ of the automorphic representation associated to $x$, and $W_x$ is the restriction to $\GL_n(F)$ of the algebraic representation of $\mathbb{G}'$ encoding the Hodge--Tate weights of the Galois
 representation attached to $x$. Thus the action of $R[1/p]$ on
 $\Hom_K( W\otimes V(\alpha, \theta), (\tilde{S}(U^{\pp}, \OO)_{\mm}\otimes_{\OO} L)^\alg)$ is semisimple and the support is equal to a finite subset of $\Sigma^{\mathrm{cl}}$ consisting of $x$ such that $W\cong W_x$ and 
 $\Hom_K(V(\alpha, \theta), \pi_{x, \pp})\neq 0$. The first condition implies that 
 $x\in \Sigma(\underline{k})$, the second condition implies that $x\in \Sigma_{sc}$ via Proposition \ref{types_super}.
  \end{proof} 

\begin{remark}
In the recent paper \cite{HMS}, Hellman, Margerin and Schraen have proved
a ``big $R$ equals big $\TT$ theorem'' in the context of definite unitary groups.
We refer to their paper for the precise details, but note here that, when their
result applies, we may rephrase Theorem~\ref{dense_hecke} as a statement
about the Zariski density of certain sets of points in the $\Spec$
of a global Galois deformation ring.
\end{remark}

\subsection{Local application}
By applying the theory of capture to $M_{\infty}$
of the six author paper \cite{six-authors}, we will get density in the support
of $M_{\infty}$.  This support will equal the full local
deformation space, provided that this latter space is irreducible.
The proof of this is the subject of the next section.

Let $R^{\square}$ be the framed deformation ring of $\rhobar: G_F \rightarrow \GL_n(k)$. If $R$ is a local noetherian $R^{\square}$-algebra 
with residue field $k$
then to $x\in \mSpec R[1/p]$ one may associate Galois representation $\rho_x: G_F\rightarrow \GL_n(\kappa(x))$ by specialising the universal framed deformation along the map $R^{\square}\rightarrow R\rightarrow \kappa(x)$. We will define the local analogs of the sets $\Sigma^{\mathrm{cl}}$, 
$\Sigma(\underline{k})$, $\Sigma_{ps}$ and $\Sigma_{sc}$ defined in the previous section. 

Let $\Sigma^{\pst}$ be the subset of $\mSpec R[1/p]$ consisting of those
$x$ for which the representation $\rho_x$ is potentially semi-stable. Let $\Sigma(\underline{k})$ be the subset of $\Sigma^{\pst}$
such that the Hodge--Tate weights of $\rho_x$ are equal to $\underline{k}$. Let $\Sigma_{ps}$ be the subset of $\Sigma^{\pst}$
such that $\rho_x$ become crystalline after a restriction to the Galois group of an abelian extension of~$F$. Alternatively $\Sigma_{ps}$ can be described 
as those $x\in \Sigma^{\pst}$ such that the Weil--Deligne representation associated to $\rho_x$ corresponds to a principal series representation via the 
classical local Langlands correspondence. Let $E$ be either a totally ramified or an unramified extension of $F$ of degree $n$ as in the previous section. We 
let $\Sigma_{sc}$ be the subset of $\Sigma^{\pst}$  consisting of those $x$ such that the Weil--Deligne representation associated to $\rho_x$ corresponds to
a supercuspidal representation $\pi_x$ via the classical local Langlands correspondence, such that $\pi_x$ contains a simple stratum $[\Aa, n', 0, \alpha]$, such that $E=F[\alpha]$ and $\alpha$ is minimal over $F$. 

Recall that in \cite{six-authors} we have constructed a local noetherian $R^{\square}$-algebra $R_{\infty}$ with residue field $k$  and an arithmetically interesting $R_{\infty}[\GL_n(F)]$-module $M_{\infty}$. We assume that 
$p$ does not divide $2n$ and that $\rhobar$ admits a potentially crystalline lift of regular weight, which is 
  potentially diagonalisable, since this is assumed in \cite{six-authors}. As noted in Remark \ref{vac}, this last assumption will become redundant in the future. 
We recall that $M_{\infty}$ is projective as an $\OO\br{K}$-module.
(This follows from \cite[Prop.~2.10]{six-authors}, which
shows that  $M_{\infty}$ is projective over $S_{\infty}\br{K}$,
where $S_{\infty}$ is a certain formal power series ring in a finite
number of variables over $\OO$.)

\begin{thm}\label{capture_local} If $\mathfrak a$ denotes
	the $R_{\infty}$-annihilator of $M_{\infty}$ 
then $\Sigma(\underline{k}) \cap \Sigma_{ps}\cap  V(\mathfrak a)$ and 
$\Sigma(\underline{k}) \cap \Sigma_{sc}\cap  V(\mathfrak a)$ are both Zariski dense 
subsets of $V(\mathfrak a)$.
\end{thm}
\begin{proof} We will treat the supercuspidal case, since
	the principal series is similar (and easier).
	The proof is essentially the same as the proof of Theorem \ref{dense_hecke}. We only need to verify that  $R_{\infty}/\mathfrak a_{\alpha, \theta}$ is reduced and all the maximal ideals of $(R_{\infty}/\mathfrak a_{\alpha, \theta})[1/p]$ lie in $\Sigma(\underline{k})\cap \Sigma_{sc}$, where 
$\mathfrak a_{\alpha, \theta}$ is the $R_{\infty}$-annihilator of $\Hom_{\OO\br{K}}^{\cont}(M_\infty, (W\otimes V(\alpha, \theta))^*)$. This is proved in the same way as 
part 1 of \cite[Lem.\,4.17]{six-authors} by using Proposition \ref{types_super} together with local--global compatibility.
\end{proof}

\begin{lem}\label{going_down} Let $A$ and $B$ be complete local noetherian $\OO$-algebras with residue field $k$. Assume that 
$B$ is a flat $A$-algebra. Let $\Sigma$ be a subset of $\mSpec A[1/p]$ and let $\Sigma'$ be its preimage in 
$\mSpec B[1/p]$. If the closure of $\Sigma'$ in $\Spec B$ is a union of irreducible components of $\Spec B$ then 
the closure of $\Sigma$ in $\Spec A$ is a union of irreducible components of $\Spec A$.
\end{lem} 
\begin{proof} Let $B_{\mathrm{tf}}$ be the maximal $p$-torsion free quotient of $B$. The map $\pp\mapsto \pp \cap  B_{\mathrm{tf}}$ induces a bijection between the minimal primes of $B[1/p]$ and 
$B_{\mathrm{tf}}$ and we will identify both sets. Moreover, every minimal prime of $B_{\mathrm{tf}}$ 
is also a minimal prime of $B$: if $\pp$ is a minimal prime of $B_{\mathrm{tf}}$ then the 
quotient $B_{\mathrm{tf}}/\pp$ is $p$-torsion free, as $p$ is regular on $B_{\mathrm{tf}}$.  If $\qq$ is a minimal prime of $B$ then either $B/\qq$ is $p$-torsion free
in which case $\qq$ is a minimal prime of $B_{\mathrm{tf}}$ or $p$ is zero in $B/\qq$ in which case 
$B/\qq$ cannot have $B_{\mathrm{tf}}/\pp$ as a quotient. Thus it is enough to study the question after inverting $p$. 

Let $\mathfrak b$ be an ideal of $B[1/p]$ such that $V(\mathfrak b)$ is the closure of $\Sigma'$ and 
let $\mathfrak a= \mathfrak b \cap A[1/p]$. Then $V(\mathfrak a)$ is equal to the closure of $\Sigma$. By assumption we may write 
$V(\mathfrak b)= V(\qq_1)\cup\ldots \cup V(\qq_r)$, where $\qq_i$ are minimal prime ideals of $B[1/p]$. We may assume that 
$\mathfrak b = \qq_1\cap \ldots \cap \qq_r$. Since $B$ is $A$-flat the going down theorem implies that $\pp_i:=\qq_i \cap A[1/p]$ is a minimal prime 
of $A[1/p]$. Since  $\mathfrak a= \mathfrak b \cap A[1/p]= \pp_1\cap \ldots \cap \pp_r$ we have $V(\mathfrak a)= V(\pp_1)\cup \ldots \cup V(\pp_r)$.
\end{proof}

\begin{cor}\label{local} Let $\mathfrak a$ be the $R_{\infty}$-annihilator of $M_{\infty}$. If $V(\mathfrak a)$ is equal to a union of irreducible components of $\Spec R_{\infty}$,
then the Zariski closure of $\Sigma(\underline{k}) \cap \Sigma_{sc}$ (resp. $\Sigma(\underline{k}) \cap \Sigma_{ps}$) in $\Spec R^\square$
contains a non-empty
union of irreducible components of $\Spec R^\square$.
\end{cor}
\begin{proof} Given Lemma \ref{going_down} and Theorem \ref{capture_local}, it is enough to verify that $R_{\infty}$ is flat over $R^{\square}$. Now 
$R_{\infty}$ is formally smooth over the ring denoted by $R^{\loc}$ in \cite[\S 2.6]{six-authors}, and $R^{\loc}= R^{\square} \wtimes_{\OO} A$, where 
$A$ is a completed tensor product over $\OO$ of some deformation rings at a finite number of places away from $\pp$. All them are $\OO$-torsion free, see
\cite[\S 2.4]{six-authors},  thus $A$ is $\OO$-torsion free and hence $\OO$-flat, which implies that $R^{\loc}$ is flat over $R^{\square}$. 
\end{proof}

The condition on $V(\mathfrak a)$ in Corollary \ref{local} will be verified in the next section. 

\section{The support of $M_{\infty}$}\label{support}
We begin with a definition.

\begin{defi}\label{define_support} If $A$ is a complete Noetherian
local $\cO$-algebra,
and $M$ is a finitely generated $A\br{K}$-module,
then we define the {\em support} of $M$ in $\Spec A$, denoted
$\supp_A M$, as follows:
$\supp_A M:= V(\mathfrak a),$
where $\mathfrak a$ is the $A$-annihilator of $M$.
\end{defi}

Since a finitely generated $A\br{K}$-module is typically not finitely generated
as an $A$-module, it is not immediately clear that defining its support
in the above manner is sensible.  However, the following lemma gives
us assurance that the preceding definition does in fact give a reasonable
notion of support.

\begin{lem}\label{all_finite} If  $m_1, \ldots, m_r$ are generators of $M$
as an $A\br{K}$-module, and $M'$ is the $A$-submodule 
of $M$ generated by $m_1, \ldots, m_r$,
then $\supp M'=\supp M$. 
\end{lem} 
\begin{proof} Let $\mathfrak a$ be the $A$-annihilator of $M$ and let $\mathfrak a'$ be the $A$-annihilator of $M'$. Since $M'\subset M$ 
we have $\mathfrak  a\subset \mathfrak a'$.
On the other hand, since the actions of $A$ and $K$ on $M$ commute,
we see that $\mathfrak a'$ annihilates $A\br{K}M' = M.$
\end{proof}  

The main goal of this section is to prove the following theorem.

\begin{theorem}
	\label{thm:support}
       	The support 
of $M_{\infty}$ in $\Spec R_{\infty}$ is equal to
a union of irreducible components of $\Spec R_{\infty}$.  
\end{theorem}
\begin{proof}
	It follows from Proposition~\ref{capture_algebraic}
	that, if $\xi$ runs over all the algebraic representations
	of $G$, then the support of $M_{\infty}$ is
	equal to the Zariski closure of the unions of the supports
	of the various modules $M_{\infty}(\xi)$.
	(See \cite{six-authors} for a discussion of these patched
	modules.)
	The support of each $M_{\infty}(\xi)$ is reduced
	\cite[Lem.\,4.17]{six-authors},
	and so it follows from Theorem~\ref{rigid support}
	below that the support of $M_\infty[1/p]$ in $\Spec R_{\infty}[1/p]$ is equal to 
	a union of its irreducible components. Since $M_{\infty}$ is $p$-torsion free 
	by construction, the support of $M_{\infty}$ is equal to  a union of  irreducible components
	of $\Spec R_{\infty, \mathrm{tf}}$, where $\Spec R_{\infty, \mathrm{tf}}$ denotes the maximal 
	$p$-torsion free quotient of $R_{\infty}$. As explained in the proof of Lemma \ref{going_down} 
	these are also irreducible components of $\Spec R_{\infty}$.
\end{proof}

\begin{remark}  Typically $R^{\square}$ 
 is a formally smooth $\cO$-algebra and thus irreducible. 
In such cases, our theorem together with Corollary \ref{local}  implies
that $M_{\infty}$ is supported on all of  $\Spf R^{\square}$.  We expect that $M_{\infty}$ is in fact
always supported on all of $\Spf R^{\square}$, but it seems {\em a priori}
possible that this is not the case, when $\Spf R^{\square}$ 
has more than one irreducible component.  
\end{remark}

\begin{theorem}
	\label{rigid support}
	The Zariski closure in $(\Spf R_{\infty})^{\rig}$
of the set of crystalline points lying in the support of $M_{\infty}$ is 
equal to
a union of irreducible components of $(\Spf R_{\infty})^{\rig}$.  
\end{theorem}

We follow the approach of Chenevier \cite{Chenevier}
and Nakamura \cite{Nakamura}, who proved (under various technical
hypotheses) that
the crystalline points are Zariski dense in the rigid analytic generic
fibre $(\Spf R^{\square})^{\rig}$
of $\Spf R^{\square}$. 

The arguments of Chenevier and Nakamura use ``infinite fern''-style
techniques.  In order to modify these arguments so as to be sensitive
to the support of $M_{\infty}$, we use the locally analytic Jacquet
module of \cite{Jacquet-one}, which provides a representation-theoretic
approach to the study of finite slope families.   Breuil, Hellmann,
and Schraen \cite{BHS} have made a careful study of the theory
of the Jacquet module as it applies to $M_{\infty}$, and we will
use their results.\footnote{In fact, they work with a slightly
different version  of $M_{\infty}$ to us (which is literally
the construction of \cite{six-authors}), but all their arguments 
and results apply equally well in our setting, and we will cite
and apply them without further discussion of this issue.}

Following the notation of \cite{BHS},
we write $\Pi_{\infty} := \Hom^{\cont}_{\cO}
(M_{\infty}, L);$ then $\Pi_{\infty}$ is an
$R_{\infty}$-admissible Banach  space representation of $G$.
We may pass to its $R_{\infty}$-locally analytic vectors, see \cite[\S 3.1]{BHS},
and then form the locally analytic Jacquet module
$J_B(\Pi_{\infty}^{R_{\infty}-\ana})$. 
The continuous dual 
$J_B(\Pi_{\infty}^{R_{\infty}-\ana})'$
of  $J_B(\Pi_{\infty}^{R_{\infty}-\ana})$
may naturally be 
regarded as the space of global sections of a coherent sheaf
$\cM_{\infty}$
on the rigid analytic 
space $(\Spf R_{\infty})^{\rig} \times \That$.
We let $X_{\infty}$ denote the support of $\cM_{\infty}$; 
then $X_{\infty}$ is a 
Zariski closed rigid analytic subspace of
$ (\Spf R_{\infty})^{\rig}\times \That$. If  $x$ is a closed point of $\Spec R_{\infty}[1/p]$ and 
 $\mm_x$ is the corresponding maximal ideal,
then $\Pi(x)$ is a closed subspace of $\Pi_{\infty}$ consisting of elements in $\Pi_{\infty}$, killed 
by $\mm_x$.

We recall again from \cite{BHS}, see their Definition~2.4,
that there is another important
Zariski closed rigid analytic subspace
$$X^{\tri}_{\infty} \hookrightarrow 
(\Spf R_{\infty})^{\rig}\times \That,$$
which by definition is the Zariski closure of the 
set of points
$$(x,\delta) \in 
(\Spf R_{\infty})^{\rig}\times \That^{\reg},$$
where $\delta$ is the parameter of a triangulation of the Galois
representation $r_x$, see \cite[\S 2.2]{BHS}.
 
We then have the following key result of \cite{BHS}.

\begin{theorem}[{\cite[Thm.\,3.21]{BHS}}]
\label{thm:BHS main}
There is an inclusion
$X_{\infty} \subset X^{\tri}_{\infty}$, 
and $X_{\infty}$ is furthermore a union of irreducible
components of $X^{\tri}_{\infty}$.
\end{theorem}

We also need to recall the key point of the infinite fern argument
of \cite{Chenevier} and \cite{Nakamura}, in a form that
is well-suited to our argument.

\begin{df} If $\underline{k}$ is a system of
labelled Hodge--Tate weights, 
then we let $R^{\underline{k}}_{\infty}$ denote the reduced and $\cO$-flat
quotient of $R_{\infty}$ characterized by the property that a 
point $x \in (\Spf R_{\infty})^{\rig}$ lies in
$(\Spf R^{\underline{k}}_{\infty})^{\rig}$ if and only if 
$r_x$ is crystalline with labelled Hodge--Tate weights given
by~$\underline{k}$.
\end{df}

Let $f:=[F_0: \Qp]$ be the degree of the maximal unramified subextension of $F$. 
\begin{defi}\label{defi:benign}
We say that a crystalline representation $r$ of $G_F$ with regular Hodge--Tate weights $\underline{k}$ 
is \textit{benign} if the following hold: 
\begin{itemize} 
\item[(i)] the eigenvalues of the linearization of the crystalline Frobenius $\varphi^f$ of $r$ are pairwise distinct;
\item[(ii)] all the refinements of $r$ are non-critical;
\item[(iii)] the ratio of any two eigenvalues in (i) is not equal to $p^{\pm f}$.
\end{itemize} 
\end{defi}
We refer the reader to \cite[\S 2.3]{BHS} for the terminology used in (i) and (ii). Our definition coincides
with the definition of benign in \cite[Def.\,2.8]{Nakamura}. It is more restrictive than \cite[Def.\,2.8]{BHS}, 
since we additionally impose the condition (iii). Let $\WD(r)$ be the Weil-Deligne representation associated to $r$. The 
condition (iii) implies that the smooth representation $\pi_p$  of $\GL_n(F)$ associated to $\WD(r)$ via the classical local
Langlands correspondence is an \textit{irreducible} unramified principal series representation.

\begin{prop}
\label{prop:benign}
If $\underline{k}$ is any regular system of labelled Hodge--Tate weights,
and $C$ is any irreducible component of
$(\Spf R_{\infty}^{\underline{k}})^{\rig}$,
then the set of points $x \in C$ for which $r_x$ is benign
forms a non-empty Zariski open subset
of $(\Spf R_{\infty}^{\underline{k}})^{\rig}$.
\end{prop}
\begin{proof}
This is a result
of Chenevier \cite[Lem.\,4.4]{Chenevier} in the case when $K = \Q_p$
and Nakamura \cite[Lem.\,4.2]{Nakamura} in the general case.
\end{proof}

\begin{theorem}
\label{thm:infinite fern}
Let $z$ be a benign crystalline point of $\Spec R_{\infty}$.
\begin{enumerate}
\item
If $(z,\delta)$ is a point of $X^{\tri}_{\infty}$ which 
lies over $z$, then $(z,\delta)$ is a smooth point
of $X^{\tri}_{\infty}$.
\item The tangent space
$T_{(\Spf R_{\infty})^{\rig}, z}$ is spanned by the
images of the tangent spaces $T_{X^{\tri}_{\infty}, (z,\delta)}$,
as $(z,\delta)$ runs over
all points of $X^{\tri}_{\infty}$ lying over $z$. 
\end{enumerate}
\end{theorem}
\begin{proof}
The first claim follows by combining
Corollary 2.12 and Theorem~2.6~(iii)
of~\cite{BHS}. 
The second claim follows from \cite{Chenevier} in the case that $K 
= \Q_p$ and \cite[Thm.\,2.62]{Nakamura2} in general once we note that the image
of 
$T_{X^{\tri}_{\infty}, (z,\delta)}$
in $T_{(\Spf R_{\infty})^{\rig}, z}$ coincides with the closed subspace of 
$T_{(\Spf R_{\infty})^{\rig}, z}$
classifing trianguline deformations of $(z,\delta)$.  See e.g.\
the discussion of this in the proof of \cite[Thm.\,2.6]{BHS}.
\end{proof}

With the preceding results recalled,
we are now in a position to study the support of $M_{\infty}$.

If $\xi$ is an irreducible algebraic representation of $G$,
and if $\underline{k}$ is the system of labelled
Hodge--Tate weights corresponding to $\xi$,
then (as already recalled in the proof of Theorem~\ref{thm:support}),
we may form the $R_{\infty}[1/p]$-module $M_{\infty}(\xi)$
as in \cite{six-authors}, whose support is a union of some
number of irreducible components of $\Spec R_{\infty}^{\underline{k}}
[1/p]$, or, equivalently,
a union of some number of irreducible components of
$(\Spf R_{\infty}^{\underline{k}})^{\rig}$.
In particular, if $M_{\infty}(\xi)$ is non-zero,
then Proposition~\ref{prop:benign} implies that its support contains a
Zariski dense open subset of points $x$ for which $r_x$ is benign.

\begin{lem}
\label{lem:refinements}
If $\xi$ is an algebraic representation of $G$,
if $x$ is a point lying in the support of $M_{\infty}(\xi)$ for which
$r_x$ is benign, and if $(x,\delta)$ is a point of $X_{\infty}^{\tri}$
lying over~$x$, then $(x,\delta) \in X_{\infty}$.
\end{lem}
\begin{proof}
We see from \cite[Cor.\,2.12]{BHS} that the $\delta$ is necessarily
a triangulation of $r_x$.  
The fibre of $\cM_{\infty}$ over $(x,\delta)$ is dual to the eigenspace
$J_B\bigl(\Pi(x)^{\ana}\bigr)^{T = \delta}$ in the Jacquet module of $\Pi(x)$, where $\ana$ denotes the subspace
of locally analytic vectors in the Banach space $\Pi(x)$, see \cite[Prop.\,3.7]{BHS}.
Thus in order to show that $(x,\delta)$ lies in $X_{\infty}$,
it suffices to show that for every parameter $\delta$ describing
a triangulation of $r_x$, the corresponding eigenspace in the Jacquet
module of $\Pi(x)^{\ana}$ is non-zero.

To see this, note that local-global compatibility shows that $\Pi(x)$
contains a locally algebraic representation $\pi_p \otimes \xi,$
with $\pi_p$ the parabolic induction of an unramified character, see \cite[Thm.\,4.35]{six-authors}. Note that the condition (iii) in Definition 
\ref{defi:benign} ensures that $\pi_p$ is irreducible and hence generic, so that the conditions of \cite[Thm.\,4.35]{six-authors} are satisfied. 
The representation $\pi_p \otimes \xi$ equipped with the finest locally convex topology is a locally analytic subrepresentation
of $\Pi(x)^{\ana}$. Since the Jacquet-functor is left exact we have an inclusion $J_B(\pi_p \otimes \xi)\subset J_B\bigl(\Pi(x)^{\ana}\bigr)$.
Moreover,  \cite[Prop.\,4.3.6]{Jacquet-one} implies that $J_B(\pi_p \otimes \xi)\cong (\pi_p)_N \otimes \xi^N$, where $N$ is the unipotent radical 
of $B$. Hence, $J_B(\pi_p\otimes\xi)$ consists of $n!$ characters which fill out
all the necessary eigenspaces.   (These correspond to the $n!$ possible
orderings of the Frobenius eigenvalues, which in turn
match with all the possible triangulations.)
\end{proof}

\begin{theorem}
The Zariski closure, in $(\Spf R_{\infty})^{\rig}$, of the
union of the supports of the modules $M_{\infty}(\xi)$, 
as $\xi$ ranges over all the irreducible algebraic representations
of $G$, is equal to a non-empty union of components
of $(\Spf R_{\infty})^{\rig}$.
\end{theorem}
\begin{proof}
To ease notation, write $Y := (\Spf R_{\infty})^{\rig}$,
and let $Z \hookrightarrow Y$ the Zariski closure under consideration in the
statement of the theorem.
We note that $M_{\infty}(\xi)$ is non-zero for at least one choice of $\xi$,
so that $Z$ is non-empty.
Since $Z$ is furthermore a reduced rigid analytic space over a field
of characteristic zero, its smooth locus $Z^{\sm}$ is a non-empty and open subset.
Thus the correspondence $V \mapsto U := V \cap Z^{\sm}$
induces a bijection between the irreducible components $V$ of $Z$ 
and the irreducible (or, equivalently, connected) components $U$ of $Z^{\sm}$; one recovers $V$  from
$U$ by taking the Zariski closure of $U$ in $Z$ (or, equivalently, in~$Y$).

We next note that the points of $X_{\infty}$ of the form $(z,\delta)$
with $z$ lying in the support of one of $M_{\infty}(\xi)$
are actually Zariski dense
in $X_{\infty}$. Indeed,
this follows from \cite[Thm.\,3.19]{BHS}, 
since  any $(z,\delta)\in X_{\infty}$, which is  strongly classical in the sense of \cite[Def.\,3.15]{BHS},  is classical 
in the sense of  \cite[Def.\,3.17]{BHS} and hence will have $z$ lying in the support of $M_{\infty}(\xi)$, where the highest weight of $\xi$ corresponds to the algebraic part of $\delta$.
This implies that the image of $X_{\infty}$ in $Y$ under the composition $\pi: X_{\infty}\subset Y\times \widehat{\mathrm{T}} \rightarrow Y$ is contained in $Z$, as $\pi^{-1}(Z)$ is a closed subset of $X_{\infty}$ containing a dense set of points.

As noted above, 
if $\underline{k}$ is the system of labelled
Hodge--Tate weights corresponding to an algebraic representation $\xi$,
then the support of $M_{\infty}(\xi)$
is a union of some number of components of
$(\Spf R_{\infty}^{\underline{k}})^{\rig}$,
each of which contains a Zariski dense subset of points
$z$ for which $r_z$ is benign.  Thus $Z$ is equal to the Zariski
closure of the set $B$ of such points, i.e.\ 
the set $B$ of points $z$ lying in the support of some $M_{\infty}(\xi)$,
and for which $r_z$ is benign.
Thus if $V$ is an irreducible component of~$Z$,
and if $U := V \cap Z^{\sm}$ (which, as we noted above, is Zariski
dense in~$V$),
then we find that the $B \cap U$ is Zariski dense in~$U$; in particular,
it is non-empty.

Choose a point $z \in B \cap U$.
If $(z,\delta)$ is any point of $X^{\tri}_{\infty}$
lying over~$z$,
the point $(z,\delta)$ lies in the smooth
locus of $X^{\tri}_{\infty}$, by Theorem~\ref{thm:infinite fern}~(1).
Lemma~\ref{lem:refinements} shows 
that any such point $(z,\delta)$
in fact lies in $X_{\infty}$. 
We then conclude from
Theorem~\ref{thm:BHS main}
that $(z,\delta)$ is a smooth point of $X_{\infty}$, and that the closed
immersion
$X_{\infty} \hookrightarrow X^{\tri}_{\infty}$ is an isomorphism
at any such point $(z,\delta)$.

Combining this last conclusion with the inclusion
$\pi(X_{\infty}) \subseteq Z$ that was proved above,
we thus find that $T_{Z,z} = T_{U,z}$ contains 
the image of the tangent space $T_{X^{\tri}_{\infty},(z,\delta)}$,
for each
choice of $(z,\delta)$ lying over $z$.
Since Theorem~\ref{thm:infinite fern}~(2) shows that
the images of the various tangent spaces $T_{X^{\tri}_{\infty},(z,\delta)}$ 
span $T_{Y,z}$,
we find that $T_{U,z} = T_{Y,z}$, and hence that $U$ and $Y$ coincide
locally at $z$.   This implies that the Zariski closure
$V$ of~$U$, which is an irreducible component of $Z$, is also
an irreducible component of~$Y$, as required.
\end{proof}

\end{document}